\documentclass[a4paper,reqno, oneside]{amsart}
\usepackage[cal=boondoxo, scr=boondox, bb=boondox]{mathalfa}
\usepackage[utf8]{inputenc}
\usepackage[backref=page]{hyperref}
\hypersetup{
pdftitle={An algebraic geometry of paths via the iterated-integrals signature},
pdfsubject={Contrary to previous approaches bringing together algebraic geometry and signatures of paths,
 we introduce a Zariski topology on the space of paths itself,
 and study path varieties consisting of all paths whose iterated-integrals signature satisfies certain polynomial equations.
 Particular emphasis lies on the role of the non-associative halfshuffle,
 which makes it possible to describe varieties of paths that satisfy certain relations all along their trajectory.
 Specifically, we may understand the set of paths on a given classical algebraic variety in R^d starting from a fixed point as a path variety,
 e.g. paths on a sphere.
 While the characteristic geometric property of halfshuffle varieties is that they are stable under stopping paths at an earlier time,
 we furthermore study varieties which are stable under the natural (semi)group operation of concantenation of paths.
 We point out how the notion of dimension for path varieties crucially depends on the fact that they may be reducible into countably infinitely many subvarieties.
 Finally, we see that studying halfshuffle varieties naturally leads to a generalization of classical algebraic curves, surfaces and affine varieties in finite dimensional space.
 These generalized algebraic sets, an example being the graph of the exponential function, are now described through iterated-integral equations.},
 pdfauthor={Rosa Lili Dora Preiß},
 pdfkeywords={path variety; shuffle ideal; halfshuffle; deconcatenation coproduct; tensor algebra; Zariski topology; concatenation of paths; Chen's identity; subpath; tree-like equivalence; regular map; generalized variety}
}

\usepackage{mathtools, mathabx, shuffle, microtype, bookmark, cleveref} \usepackage[hmargin={2cm,4cm}]{geometry}
\usepackage{parskip}

\usepackage[usenames,dvipsnames,table]{xcolor}
\usepackage{url}

\usepackage[linecolor=white,backgroundcolor=white,bordercolor=white,textsize=tiny]{todonotes}
\let\todon\todo
\renewcommand{\todo}[1]{\todon{\color{red}#1}}

\newtheorem{theorem}{Theorem}[section]
\newtheorem{definition}[theorem]{Definition}
\theoremstyle{plain}

\newtheorem{conjecture}[theorem]{Conjecture}
\newtheorem{corollary}[theorem]{Corollary}

\newtheorem{example}[theorem]{Example}

\newtheorem{lemma}[theorem]{Lemma}

\newtheorem{proposition}[theorem]{Proposition}
\newtheorem{remark}[theorem]{Remark}

\newtheorem{question}[theorem]{Question}

\newcommand{\R}{\mathbb{R}}

\newcommand{\C}{\mathbb{C}}

\newcommand{\spann}{\operatorname{span}} 
\renewcommand{\Im}{\operatorname{Im}}

\newcommand\proj{\operatorname{proj}}

\newcommand{\id}{\operatorname{id}}

\def\word#1{{\color{blue}\mathbf{#1}}}

\newcommand{\conc}{\mathbin{\bullet}}

\newcommand{\emptyword}{\mathsf{e}}
\newcommand{\varltwo}[1]{C^{2^-\text{-var}}(#1)}
\newcommand{\varltwogroup}[1]{\check{C}^{2^-\text{-var}}(#1)}
\newcommand{\varltwoz}[1]{C_0^{2^-\text{-var}}(#1)}
\newcommand{\pathsin}[1]{P_{\in #1}}
\newcommand{\antipode}{\mathcal{A}}

\newcommand{\incin}[1]{E_{\in #1}}
\newcommand{\loops}[1]{\mathcal{L}_{#1}}
\newcommand{\treelike}{\mathrel{\stackrel{\text{tree}}{\cong}}}

\tikzstyle{endpoints}=[fill=black, draw=black, shape=circle, inner sep=0.15ex]

\tikzstyle{dashededge}=[-, dashed]
\tikzstyle{arrowedge}=[->]
\tikzstyle{rededge}=[-, draw=red]
\tikzstyle{bluearrowedge}=[->, draw=blue]

\let\origmaketitle\maketitle
\def\maketitle{
  \begingroup
  \let\MakeUppercase\relax \origmaketitle
  \endgroup
}

\title{An algebraic geometry of paths \mbox{via the iterated-integrals signature}}
\date{\today}
\author{\Large Rosa Lili Dora Preiß$^{\includegraphics[width=0.02\textwidth]{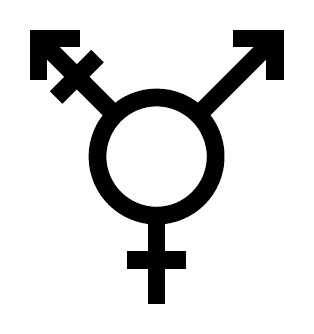}}$}
\thanks{$^{\includegraphics[width=0.015\textwidth]{trans}}$MPI MiS Leipzig}

\begin{document}
\maketitle
\begin{abstract}
 Contrary to previous approaches bringing together algebraic geometry and signatures of paths,
 we introduce a Zariski topology on the space of paths itself,
 and study path varieties consisting of all paths whose iterated-integrals signature satisfies certain polynomial equations.
 Particular emphasis lies on the role of the non-associative halfshuffle,
 which makes it possible to describe varieties of paths that satisfy certain relations all along their trajectory.
 Specifically, we may understand the set of paths on a given classical algebraic variety in $\R^d$ starting from a fixed point as a path variety,
 e.g.\ paths on a sphere.
 While the characteristic geometric property of halfshuffle varieties is that they are stable under stopping paths at an earlier time,
 we furthermore study varieties which are stable under the natural (semi)group operation of concantenation of paths.
 We point out how the notion of dimension for path varieties crucially depends on the fact that they may be reducible into countably infinitely many subvarieties.
 Finally, we see that studying halfshuffle varieties naturally leads to a generalization of classical algebraic curves, surfaces and affine varieties in finite dimensional space.
 These generalized algebraic sets, an example being the graph of the exponential function, are now described through iterated-integral equations.\\[2ex]
\end{abstract}

\tableofcontents
\setcounter{section}{-1}
\section{Introduction}

It is well established that the iterated-integrals signature provides a powerful connection between algebra and geometry.
The correspondence between the antipode and time inversion of the path, as well as between the internal tensor product of the tensor algebra and concatenation of paths, already highlights that,
and hints towards the geometry-algebra variety-ideal correspondence we are introducing in this paper.
Furthermore, we know from previous work that certain components of the signature have a very nice geometrical interpretation, most notably maybe the signed areas and signed volumes.
In \cite{DR18}, it was shown that the signed areas and signed volumes can for piecewise linear paths be expressed as sums over areas/volumes of simplices spanned by the linear pieces, and that for the moment curve (but not for every path), the absolute value of the signed volume is exactly the volume of the convex hull.
Then \cite{amendolaleemeroni23} showed that this special property holds for all cyclic paths, i.e. paths which are the limit of paths which satisfy a certain strict multidimensional convexity condition in terms of a non-negative determinant for all combinations of distinct intermediate points.

In \cite{DLPR20}, it has been shown that the signature is polynomial in the iterated signed areas, providing further evidence of its intrinsic geometrical nature.

With \cite{AFS18}, for the probably first time, algebraic geometry and the study of iterated-integral signatures have been brought together in a practical way ready for computations. They look at the complex projective Zariski closure of the image of the truncated signature map for piecewise linear paths and polynomial paths.
In \cite{G18}, complex projective signature image varieties are studied for log-linear rough paths.
In \cite{CGM20}, it is shown that these varieties are toric for the classes of log-linear and axis parallel paths.

In contrast to these previous approaches of the last years,
we do not want to restrict ourselves to studying the geometry of the image of the truncated signature in finite dimensional spaces.
Instead, we want to study a Zariski topology on the infinite dimensional path space itself.

An inspirational example of what we will consider a path variety is the orbit of a path under pointwise action by a compact linear group.
This object to be described as the solution set of linear/polynomial equations in the signature has been conjectured in \cite{DR18},
and proven in \cite{diehllyonsnipreiss}.

In \cite{diehlreizensteinpreiss} conjugation, loop and closure invariants are studied,
and this straightforward leads  to the question whether the orbit of a path up to conjugation or closure is an algebraic set.

Furthermore, in \cite[Corollary 3]{colmenarejopreiss20} it was understood that for a polynomial map $p:\,\mathbb{R}^d\to\R^t$,
the set of all paths $X$ such that $p(X)$ is tree-like is described by linear/polynomial equations in the signature,
and thus it also forms a path variety.

In fact,
already in \cite[Section~3]{bib:Che1954} the path varieties $\mathcal{L}_k$ of loops of order $k$ (loops with signature vanishing up to order $k$) have been introduced
not as varieties,
but as subgroups of the reduced path group.
We study general algebraic subgroups in Section \ref{sec:subsemigroups}.

Many intricacies of the study of path varieties compared to classical algebraic geometry
result from the fact that even though the image of the signature lies dense in the free Lie group,
it is a strict and hard to describe subgroup.
Indeed, a classical argument by Lyons is that any element in the free Lie group has a square root,
as the exponential map from free Lie series to the free Lie group is bijective,
but many paths, such as a semicircle, the sinus graph from $0$ to $2\pi$, the concatenation of two non-parallel linear paths,
cannot have a square root under path concatenation.

\subsection{Outline}

In section \ref{sec:topology}, we introduce our new Zariski topology on path space and
state as a first example that the set of paths with increment lying in a subset $M\subset\R^d$ is a path variety if and only if $M$ is a point variety.
We then discuss what the fact that the signature characterizes the path
only up to tree-like equivalence
means for the Zariski topology,
how our setting is a necessarily infinite dimensional one,
and further differences to what we know from algebraic geometry on finite dimensional vector spaces.
Finally, we observe how path varieties contain limits of analytically converging sequences of paths.

In section \ref{sec:conc}, we understand the concatenation of paths as a Zariski continuous map,
and give explicit formulas for the preimage varieties.

Section \ref{sec:halfshuffle} is devoted to our first main result,
stating that varieties stable under stopping reduced paths at an earlier time
correspond to halfshuffle ideals.
An immediate corollary using the antipode/time reversal correspondence is that varieties stable under starting reduced paths at a later time are those defined by ideals of the opposite halfshuffle.
We then state the important consequence that reduced paths starting from $0$ and then staying on a fixed point variety can be described by a path variety.

Section \ref{sec:rank_varieties} introduces a class of halfshuffle varieties which as we show includes the set of reduced paths in some finite dimensional subspace of given dimension,
as well as the set of reduced paths lying in some sub point variety of $\R^d$ of given degree.

In Section \ref{sec:homogeneous} we show that homogeneous ideals correspond to path varieties that are stable under rescaling.
This makes it possible to understand the description of finite dimensional path spaces presented in \cite{AFS18} within our new framework.

In Section \ref{sec:subsemigroups} we will understand the correspondence of path varieties stable under concatenation with Hopf ideals.
The results Theorem \ref{thm:infdivi}, Lemma \ref{lem:commutation} and Theorem \ref{thm:irrationalroot} characterizing infinitely divisible paths are of general relevance for the study of the reduced path group.

In section \ref{sec:dim}, we describe the notion of dimension for path varieties,
which relies on the concept of countably irreducible varieties,
and see how path varieties with finitely generated ideals are always infinite dimensional.

In section \ref{sec:morphisms}, we introduce regular maps between path spaces,
which are in particular Zariski continuous.
We then look at the particular examples of $M_p$, which are pointwise polynomial for the path,
and $\Lambda_B$, which are pointwise polynomial for the signature.
Finally, we discribe a more general class of regular maps which is stable under composition
and includes $\Lambda_B$, time inversion, concatenation powers and concatenation of a fixed path.

We see in section \ref{sec:generalized} how the family of $M\subseteq\R^d$ such that the set of reduced paths starting in $0$ and staying in $M$ is a path variety, is more general than just point varieties.

\subsection{Acknowledgements}

We thank Carlos Améndola, Simon Breneis, Ilya Chevyrev, Joscha Diehl, Kurusch Ebrahimi-Fard, Peter Friz, Francesco Galuppi, Xi Geng, Leonie Kayser, Darrick Lee, Felix Lotter, Ludwig Rahm, Jeremy Reizenstein, Alexander Schmeding, Leonard Schmitz and Bernd Sturmfels for inspiring discussions and suggestions.
We in particular thank Carlos Améndola and Bernd Sturmfels for suggesting that the $M_p$ can be understood as morphisms of path varieties,
for comments on the geometry of points with no history versus the geometry of paths with a history,
for proofreading,
as well as asking the question of whether the set of paths lying in some surface of given degree can be described as a path variety.
We also specifically thank Felix Lotter and Leonard Schmitz for all the conversations while working on joint projects related to the topic of this article.
We furthermore thank Ludwig Rahm for the discussion of a talk on the present article during an Oberwolfach Mini-Workshop, see \cite{Oberwolfach23}.

\subsection{Basic definitions for and results on signatures}
 Our study of the geometry of path space crucially relies on two topological operations for paths,
 namely the associative concatenation of two paths
 and the involutive time reversal.

 For continuous paths $X,Y$ through $\mathbb{R}^d$, let $X\sqcup Y$ denote their concatenation,
 i.e. $X$ continuously followed by $Y$ with $Y$ translated such that the starting point of $Y$ then falls onto the end point of $X$,
 or in formulas,
 for $[0,T_1]$ the time domain of $X$ and $[0,T_2]$ the time domain of $Y$,
 \begin{equation*}
  (X\sqcup Y)_t:=
  \begin{cases}
   X_t,& t\in[0,T_1],\\
   Y_{t-T_1}+X_{T_1}-Y_0,& t\in(T_1,T_1+T_2].
  \end{cases}
 \end{equation*}

 For a path $X:[0,T]\to\R^d$, let its time reversal
 $\overleftarrow{X}$ given by $X$ run backwards,
 i.e. $\overleftarrow{X}_t:=X_{T-t}$.

 These two operations allow us to define an important equivalence relation between paths.
 To describe it,
 an $\mathbb{R}$-tree $(\tau,d)$ is defined as a metric space such that for any two distinct $x,y\in\tau$ there is an isometry $Z:\,[0,d(x,y)]\to\Im Z\subseteq\tau$ which is the up to reparametrization unique continuous map $[0,d(x,y)]$ to $\tau$ such that $Z(0)=x$ and $Z(d(x,y))=y$ (cf.\ \cite[Definition~2.1]{hamblylyons2008}).
 Following \cite[Definition~1.1]{BGLY16}, a path $X:[0,T]\to\R^d$ is then said to be tree-like if there is an $\mathbb{R}$-tree $\tau$, a continuous function $f:\,[0,T]\to\tau$ with $f(0)=f(T)$
and a function $g:\,\tau\to\mathbb{R}^d$ such that $X=g\circ f$ ($g$ is then automatically continuous on the image of $f$).
Finally, two paths $X$ and $Y$ are called tree-like equivalent, denoted as $X\treelike Y$, if $X\sqcup\overleftarrow Y$ is tree-like (\cite[Section~2]{BGLY16}).
For an example see Figure \ref{fig:treelike}.

\begin{figure}
\includegraphics[width=0.8\textwidth]{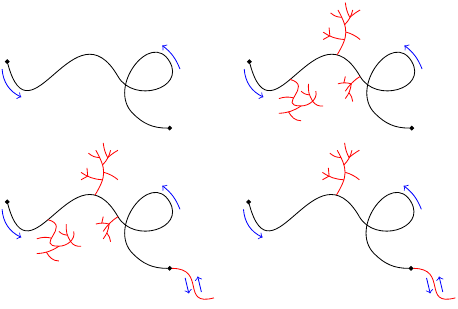}
\caption{Four tree-like equivalent paths with start- and endpoint highlighted. The upper left path is the reduced path (see Section \ref{sec:difftoclassical}).}\label{fig:treelike}
\end{figure}

While these notions so far were well-defined for all continuous paths,
for computing Stieltjes integrals,
we need to restrict their roughness.

 So, let us consider the space $\varltwo{\R^d}$ of continuous paths in $\mathbb{R}^d$ which have finite $p$-variation for some $p<2$, i.e. following \cite[Deﬁnition~5.1]{FrizVictoir10}
\begin{equation*}
 \varltwo{\R^d}:=\{X:[0,T]\to\R^d|T\in (0,\infty),\, X\text{ continuous},\,\exists p\in[1,2):\|X\|_{p\text{-var}}<\infty\}
\end{equation*}
where
\begin{equation*}
  \|X\|_{p\text{-var}}:=\Bigg(\sup_{P\in\mathcal{P}([0,T_X])}\,\sum_{[s,t]\in P}\|X_t-X_s\|_2^p\Bigg)^{1/p},
\end{equation*}
with $\mathcal{P}([0,T_X])$ the set of partitions of the time domain $[0,T_X]$ of $X$.

Note that $\varltwo{\R^d}$ is not a vector space because the addition is not well-defined due to the differing time intervals.
We could make it one by rescaling everything to a standard time interval $[0,1]$,
but we conceptually choose not to see our path space as a vector space
but rather as a semigroup with product the concatenation $\sqcup$
and additionally an involution given by the time inverse $X\mapsto\overleftarrow{X}=\antipode X$,
and rescaling all paths to $[0,1]$ would destroy the associativity of $\sqcup$.

A finite $p$-variation for some $p<2$ ensures Riemann-Stietljes integrability of the components against each other (a special case of the results in \cite{young36}),
and existence of our object of study,
the iterated-integrals signature $\sigma:\varltwo{\R^d}\to T((\R^d))$,
where $T((\R^d))$ is the linear dual space of the tensor algebra $T(\R^d):=\bigoplus_{n=0}^\infty (\R^d)^{\otimes n}$.
We identify the tensor algebra with the space of words
\begin{equation*}
\spann(\{\emptyword\}\cup\{\word{i}_1\cdots\word{i}_n|n\in\mathbb{N},\,\word{i}_j\in\{\word{1},\dots,\word{d}\}\}),
\end{equation*}
where the empty word $\emptyword$ spans $\R=(\R^d)^{\otimes 0}$
and the words $\word{i}_1\cdots\word{i}_n$ in arbitrary letters from the alphabet $\{\word{1},\dots,\word{d}\}$
span the level $n$ tensors $(\R^d)^{\otimes n}$.
The signature of the path $X$ is then defined by assigning to a word $\word{i}_1\cdots\word{i}_n$ the corresponding iterated Riemann-Stieltjes integral of the components $X^{[i_1]},\dots,X^{[i_n]}$,
i.e.
\begin{equation*}
 \langle\sigma(X),\word{i}_1\dots\word{i}_n\rangle:=\int_{0<t_1<\dots< t_n<T} \mathrm{d}X^{[i_1]}_{t_1}\dots\mathrm{d}X^{[i_n]}_{t_n}.
\end{equation*}
  We furthermore put $\sigma(X)_t:=\sigma(X|_{[0,t]})$ for the signature stopped at an earlier time $t\in[0,T]$.

  \begin{figure}
   \includegraphics[width=0.85\textwidth]{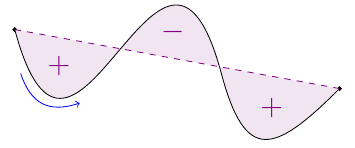}
   \caption{The signed area of the example path is the sum of the individual area segments counted with the displayed presigns.}\label{fig:areaex}
  \end{figure}

  Considering for example a path $X\in\varltwo{\R^2}$,
  we have the increments
  \begin{equation*}
   \langle\sigma(X),\word{1}\rangle=\int_0^T \mathrm{d} X^{[1]}_t=X^{[1]}_T-X^{[1]}_0,\quad \langle\sigma(X),\word{2}\rangle=\int_0^T \mathrm{d} X^{[2]}_t=X^{[2]}_T-X^{[2]}_0
  \end{equation*}
  and its signed area (see Figure \ref{fig:areaex}) is given by
  \begin{equation*}
   \langle\sigma(X),\tfrac{1}{2}(\word{12}-\word{21})\rangle.
  \end{equation*}

  The iterated-integrals signature goes back to the works of Chen \cite{bib:Che1954}, \cite{C57}, \cite{chen1958}
  and experienced a revival in stochastic analysis through Lyons' rough paths \cite{lyons94}, \cite{lyons95}, \cite{Lyons98}.

  To describe some very useful algebraic relations that the iterated-integrals signature satisfies,
  we need to first describe the algebraic structure associated with the tensor algebra.
  For an introduction to Hopf algebras see for example \cite{manchon06}, \cite{preiss16}, \cite{abe80} or \cite{sweedler69}.
  For the specific structure on the tensor algebra,
  we furthermore refer to \cite{reutenauer93}.
  For the halfshuffle,
  see e.g. \cite{foissypatras2020}, \cite{bib:FP13} or \cite{colmenarejopreiss20}.

 Let $\conc:T(\mathbb{R}^d)\times T(\mathbb{R}^d)\to T(\mathbb{R}^d)$ denote the bilinear associative concatenation in the tensor algebra,
 \begin{equation*}
 \word{i}_1\dots\word{i}_n\conc\word{i}_{n+1}\dots\word{i}_{n+m}:=\word{i}_1\dots\word{i}_n\word{i}_{n+1}\dots\word{i}_{n+m}
 \end{equation*}
 i.e. the internal tensor product.
 We also identify $T((\R^d))$ with the space of infinite formal series over words,
 and extend $\conc$ to $T((\R^d))$.
 Let $\Delta_{\conc}:\,T(\R^d)\to T(\R^d)\otimes T(\R^d)$ denote the deconcatenation,
 i.e.\ the linear coassociative coproduct dual to $\conc$,\begin{equation*}
  \langle x\conc y,z\rangle=\langle x\otimes y,\Delta_{\conc} z\rangle
 \end{equation*}
 for all $x,y\in T((\R^d))$ and all $z\in T(\R^d)$,
 where on the right hand side we have the pairing between
 $T((\R^d))\otimes T((\R^d))$ and $T(\R^d)\otimes T(\R^d)$
 given by $\langle\sum_{i}x_i\otimes y_i,\sum_{j}v_i\otimes w_i\rangle
 :=\sum_{i,j}\langle x_i,v_i\rangle\langle y_i,w_i\rangle$. For example,
 \begin{equation*}
  \Delta_{\conc} \word{1234}=\emptyword\otimes\word{1234}+\word{1}\otimes\word{234}+\word{12}\otimes\word{34}+\word{123}\otimes\word{4}+\word{1234}\otimes\emptyword.
 \end{equation*}
 Coassociativity is the property dual to associativity characterized by $(\Delta_{\conc}\otimes \id)\Delta_{\conc}=(\id\otimes\Delta_{\conc})\Delta_{\conc}.$

 We furthermore introduce Sweedler's notation for the coproduct.
 For any bilinear $B:\,T(\R^d)\times T(\R^d)\to W$
 for some vector space $W$,
 by the universal property of the tensor product there is a unique linear $\mathrm{m}_B:\,T(\R^d)\otimes T(\R^d)\to W$
 with $\mathrm{m}_B(x\otimes y)=B(x,y)$ for all $x,y\in T(\R^d)$.
 We then write
 \begin{equation*}
  \sum_{(x)}^{\conc}B(x_1,x_2):=\mathrm{m}_B\Delta_{\conc}x.
 \end{equation*}
 Similarly,
 if $\mathrm{m}_M:\,T(\R^d)^{\otimes n}\to W$ is the linear map corresponding to the multilinear map
 $M:\,T(\R^d)^{\times n}\to W$,
 then
 \begin{equation*}
  \sum_{(x)}^{\conc n}M(x_1,\dots,x_n):=\mathrm{m}_M\Delta^{n-1}_{\conc}x
 \end{equation*}
 with $\Delta^1_{\conc}:=\Delta_{\conc}$ and $\Delta^{k+1}_{\conc}:=(\id\otimes \Delta_{\conc}^k)\Delta_{\conc}$.
 Due to coassociativity,
 we have $(\id\otimes \Delta_{\conc}^k)\Delta_{\conc}=(\Delta_{\conc}^k\otimes\id)\Delta_{\conc}$ for all $k$.

 Let the bilinear right $\succ$ and left $\prec$ halfshuffles $\succ,\prec:T^{\geq 1}(\mathbb{R}^d)\times T^{\geq 1}(\mathbb{R}^d)\to T^{\geq 1}(\mathbb{R}^d)$ be recursively defined by
 \begin{align*}
  w\succ\word{i} &:= w\word{i},&\word{i}\prec w &:= \word{i}w\\
  w\succ v\word{i} &:= (w\succ v+v\succ w)\word{i},& \word{i}v \prec w&:= \word{i}(w\prec v+v\prec w)
 \end{align*}
where $w,v$ are non-empty words and $\word{i}$ is a letter.
For example,
\begin{align*}
 \word{12}\succ\word{34}&=\word{1234}+\word{1324}+\word{3124},\\
 \word{12}\prec\word{34}&=\word{1234}+\word{1324}+\word{1342}.
\end{align*}

Halfshuffles go back to Eilenberg-Mac-Lane \cite{EM53} and Schützenberger \cite{S58}.

Furthermore, let
\begin{equation}\label{eq:shsuccprec}
x\shuffle y :=x\succ y+y\succ x=x\prec y+y\prec x
\end{equation}
denote the commutative associative shuffle product, and we compute
\begin{equation*}
 \word{12}\shuffle\word{34}=\word{1234}+\word{1324}+\word{3124}+\word{1342}+\word{3142}+\word{3412}.
\end{equation*}

The halfshuffles themselves are non-commutative and non-associative,
but they satisfy the left resp.\ right Zinbiel identities
\begin{align*}
 (x\prec y)\prec z&=x\prec (y\prec z+z\prec y),\\
 x\succ (y\succ z)&=(x\succ y+y\succ x)\succ z.
\end{align*}
Note that this structure should not be confused with a non-commutative dendriform algebra!

Additionally,
we introduce a linear map $T(\R^d)\to T(\R^d)$
which we will see corresponds to time reversal,
\begin{equation*}
 \antipode(\word{i}_1\dots\word{i}_n):=(-1)^n\word{i}_n\dots\word{i}_1.
\end{equation*}
$(T(\R^d),\shuffle,\Delta_{\conc})$ is then a Hopf algebra together with the antipode $\antipode$.
In any Hopf algebra, the antipode is uniquely characterized by the antipode property\begin{equation*}
 \mathrm{m}_{\shuffle}(\antipode\otimes\id)\Delta_{\conc}w=\langle \emptyword,w\rangle \,\emptyword,
\end{equation*}
where $\mathrm{m}_{\shuffle}(w\otimes v):=w\shuffle v$.
Note that in Sweedler's notation,
\begin{equation*}
 \mathrm{m}_{\shuffle}(\antipode\otimes\id)\Delta_{\conc}w=\sum_{(w)}^{\conc}\antipode w_1\shuffle w_2.
\end{equation*}

Coming back to the halfshuffles, since the definitions of $\succ$ and $\prec$ are mirror images of each other, we then get
\begin{equation}\label{eq:antipodehalfshuffle}
\mathcal{A}(x\succ y)=\mathcal{A}x\prec \mathcal{A} y\text{ and }\mathcal{A}(x\prec y)=\mathcal{A}x\succ \mathcal{A} y,
\end{equation}
and in particular $\antipode$ is a shuffle homomorphism.

For all $X\in\varltwo{\R^d}$, $X:\,[0,T]\to\mathbb{R}^d$ and $x,y\in T^{\geq 1}(\mathbb{R}^d)$, we have the halfshuffle identity
\begin{equation*}
 \langle\sigma(X),x\succ y\rangle=\int_0^T \langle\sigma(X)_t,x\rangle\,\mathrm d\langle\sigma(X)_t,y\rangle,
\end{equation*}
where integration is with respect to the $t$ variable,
and this implies Ree's shuffle identity
\begin{equation*}
\langle\sigma(X),x\shuffle y\rangle=\langle\sigma(X),x\rangle\langle\sigma(X),y\rangle.
\end{equation*}
The halfshuffle identity can also be seen through the original proof of Ree's shuffle identity in \cite{Ree58}, but up to the author's knowledge was explicitly mentioned only much later (see e.g.\ \cite{GK08}).

In light of the halfshuffle identity,
$x\shuffle y=x\succ y+y\succ x$ encodes integration by parts,
\begin{equation*}
 \int_0^T f(t) dg(t)+\int_0^T g(t) df(t)=f(T)g(T)-f(0)g(0),
\end{equation*}
where the last term vanishes for signature components corresponding to $x,y\in T^{\geq 1}(\mathbb{R}^d)$ because $\sigma(X)_0=\emptyword$.
Similarly, the right Zinbiel identity is nothing but an algebraic encoding of
\begin{equation}\label{eq:anazinbiel}
 \int_0^T f(t) d\Big(\int_0^t g(s)dh(s)\Big)=\int_0^T f(t)g(t) dh(t)
\end{equation}
for $f,g,h\in\varltwo{\R}$.

Because of the shuffle identity, the signature takes values in the grouplike elements
\begin{equation*}
\mathcal{G}_d:=\{g\in T((\mathbb{R}^d))|g\neq 0,\langle g,x\shuffle y\rangle=\langle g,x\rangle\langle g,y\rangle\,\forall x,y\in T(\R^d)\}.
\end{equation*}
The grouplike elements form a group with product $\conc$
and inverse given by the adjoint of $\antipode$,
which we will also denote by $\antipode$
since it happens to have the same formula on words.
$\mathcal{G}_d$ is an infinite dimensional Lie group with Lie algebra given by
\begin{equation*}
\mathfrak{g}((\R^d)):=\{v\in T((\mathbb{R}^d))|\langle v,x\shuffle y\rangle=\langle v,x\rangle\langle \emptyword,y\rangle+\langle \emptyword,x\rangle\langle v,y\rangle\,\forall x,y\in T(\R^d)\}.
\end{equation*}
$\mathfrak{g}((\R^d))$ is the closure in $T((\R^d))$ of the free Lie algebra $\mathfrak{g}(\R^d)$, i.e.\ the Lie algebra with Lie bracket $[x,y]:=x\conc y-y\conc x$ generated by the letters.
The Lie-theoretic exponential map $\exp_{\conc}$ here has the particular property that it is bijective from $\mathfrak{g}((\R^d))$ to $\mathcal{G}_d$.
Denoting its inverse by $\log_{\conc}$,
we have
\begin{equation*}
 \exp_{\conc}(v)=\sum_{n=0}^\infty \frac{x^{\conc n}}{n!},\quad \log_{\conc}(g)=\sum_{n=1}^\infty (-1)^{n+1}\,\frac{(g-\emptyword)^{\conc n}}{n},
\end{equation*}
where $x^{\conc 0}:=\emptyword$.

Linear functionals of $\log_{\conc}(g)$ can be expressed as linear functionals in $g$ via coordinates of the first kind, i.e.\ for any $w\in T(\R^d)$, there is a coordinate of the first kind $c_w$ with
\begin{equation*}
 \langle\log_{\conc}(g),w\rangle=\langle g,c_w\rangle.
\end{equation*}
Coordinates of the first kind form a linear subspace of the tensor algebra,
and they are also characterized by the property
\begin{equation*}
 \langle g^{\conc n},c_w\rangle=n\,\langle g,c_w\rangle\quad\forall\, n\in\mathbb{N},\,g\in\mathcal{G}_d.
\end{equation*}

For the signature, we have Chen's identity \cite[Theorem~3.1]{bib:Che1954}
\begin{equation*}
\sigma(X\sqcup Y)=\sigma(X)\conc \sigma(Y),
\end{equation*}
which can dually be written as
\begin{equation*}
 \langle\sigma(X\sqcup Y),w\rangle=\langle\sigma(X)\otimes\sigma(Y),\Delta_{\conc} w\rangle.
\end{equation*}
Consequently,
\begin{equation*}
 \sigma(\overleftarrow{X})=\sigma(X)^{\conc -1}=\antipode\sigma(X),
\end{equation*}
and this is why we will also put $\antipode X:=\overleftarrow{X}$.

The Chen-Chow theorem in the version of \cite{FrizVictoir10}[Theorem 7.28] states that any element in the $k$-level truncated $\proj_{\leq k}\mathcal{G}_d$ is equal to a $k$-level truncated signature $\proj_{\leq k}\sigma(X)$ of a piecewise linear path $X$.

Finally, we fix some further notation for the article.

Let $\mathscr{I}_{\shuffle}(W)$ denote the shuffle ideal generated by $W\subseteq T(\mathbb{R}^d)$, i.e.\ the smallest subspace such that
\begin{equation*}
 x\in\mathscr{I}_{\shuffle}(W),\,y\in T(\mathbb{R}^d)\,\Rightarrow\,x\shuffle y \in \mathscr{I}_{\shuffle}(W).
\end{equation*}
By $\mathscr{I}_{\succ}(W)$ we denote the two-sided ideal with respect to the (non-associative!) right halfshuffle $\succ$,
i.e.\
\begin{equation*}
 x\in\mathscr{I}_{\succ}(W),\,y\in T(\mathbb{R}^d)
 \,\Rightarrow\,x\succ y, \, y\succ x \in \mathscr{I}_{\succ}(W).
\end{equation*}
Because of \eqref{eq:shsuccprec}, $\mathscr{I}_{\succ}(W)$ is then automatically a shuffle ideal too.
$\mathscr{I}_{\prec}(W)$ is defined analogously as the two sided ideal with respect to $\prec$.
Lastly, $\mathscr{I}_{\succ,\prec}(W)$ is defined as the two-sided ideal with respect to both $\prec$ and $\succ$, i.e.\
\begin{equation*}
 x\in\mathscr{I}_{\succ,\prec}(W),\,y\in T(\mathbb{R}^d)
 \,\Rightarrow\,x\succ y,\, y\succ x,\, x\prec y,\, y\prec x \in \mathscr{I}_{\succ,\prec}(W).
\end{equation*}
In particular,
\begin{equation*}
 \mathscr{I}_{\shuffle}(W)\subseteq\mathscr{I}_{\succ}(W)\subseteq\mathscr{I}_{\succ,\prec}(W),\quad \mathscr{I}_{\shuffle}(W)\subseteq\mathscr{I}_{\prec}(W)\subseteq\mathscr{I}_{\succ,\prec}(W)
\end{equation*}

In the following example for $T(\R^2)$,
we give a basis for the ideals up to level $2$,
\begin{align*}
 \mathscr{I}_{\shuffle}(\word{1212})&=\spann\{\word{1212},\,2\cdot\word{11212}+2\cdot\word{12112}+\word{12121},\,\word{21212}+2\cdot\word{12212}+2\cdot\word{12122},\,\dots\},\\
 \mathscr{I}_{\succ}(\word{1212})&=\spann\{\word{1212},\,\word{11212}+\word{12112}+\word{12121},\,\word{12121},\,\word{21212}+2\cdot\word{12212},\,\word{12122},\,\dots\},\\
 \mathscr{I}_{\prec}(\word{1212})&=\spann\{\word{1212},\,2\cdot\word{12112}+\word{12121},\,\word{11212},\,\word{12212}+\word{12122},\,\word{21212},\,\dots\},\\
 \mathscr{I}_{\succ,\prec}(\word{1212})&=\spann\{\word{1212},\,\word{12112},\,\word{12121},\,\word{11212},\,\word{12212},\,\word{12122},\,\word{21212},\,\dots\}.
\end{align*}

Let $\Lambda_B: T(\mathbb{R}^m)\to T(\mathbb{R}^n)$ for $B:\,\R^m\to T(\R^n)$ linear be recursively defined by $\Lambda_B\emptyword=\emptyword$, $\Lambda_B \word{i}=B(\word{i})$ and $\Lambda_B w\word{i}=\Lambda_B w\succ \Lambda_B\word{i}$.
For $p:\mathbb{R}^n\to \mathbb{R}^m$ polynomial with $p(0)=0$, $\varphi(p_i)$ is given by the image of $p_i$ under the unique homomorphism from $\mathbb{R}[x_1,\dots x_n]$ to the shuffle algebra which maps $x_j$ to $\word{j}$.
Put $M_p:=\Lambda_{\varphi(p)}$ for $\varphi(p)(\word{i}):=\varphi(p_i)$.

\section{Path Zariski topology}
\label{sec:topology}
In classical algebraic geometry, affine subvarieties of $\R^d$
 are sets of the form (see e.g.\ \cite[Definition 2.1.1.]{bochnakcosteroy})
 \begin{equation*}
  V(P)=\{x\in\R^d|p(x)=0\,\forall p\in P\}
 \end{equation*}
 where $P$ is a set of polynomials $p:\,\R^d\to\R$.

 Similarly, we now consider varieties in $\varltwo{\R^d}$.

The new concept we are introducing and studying in this paper is the following.
We call a \emph{path variety} any set of the form
\begin{equation*}
\mathcal{V}(W):=\{X\in\varltwo{\R^d}|\langle\sigma(X),x\rangle=0\,\forall x\in W\},
\end{equation*}
where $W$ is any subset of the tensor algebra.

Note that it is part of the folklore of the study of iterated-integrals signatures that polynomial functions in the signature are linear functions in the signature.
Indeed, any component of the signature can be seen as a polynomial over the log-signature.
So with our construction, we are really connecting path space with a ring of (shuffle) polynomials.

Path varieties are in 1-to-1 correspondence to the $2^-\text{-var}$ radical\footnote{Note that we don't have a Nullstellensatz yet here,
so no algebraic characterization of what it means to be a  $2^-\text{-var}$ radical.
It can however easily be seen to be a stronger property than being a real radical.}
shuffle ideals
\begin{equation*}
\mathcal{I}(M):=\{x\in T(\mathbb{R}^d)|\langle\sigma(X),x\rangle=0\,\forall X\in M\}, \quad M\subseteq\varltwo{\R^d}.
\end{equation*}
Indeed, by the shuffle identity, it is immediately evident that $\mathcal{I}(M)$ is a shuffle ideal,
i.e. a subspace $I$ of $T(\R^d)$ such that $x\shuffle y\in I$ for all $x\in I$ and $y\in T(\R^d)$.
Furthermore $\mathcal{V}(W)=\mathcal{V}(\mathcal{I}(\mathcal{V}(W))$ as well as $\mathcal{I}(M)=\mathcal{I}(\mathcal{V}(\mathcal{I}(M))$,
so all path varieties are obtained from $2^{-}\text{-var}$ shuffle ideals under $\mathcal{V}$
and all $2^{-}\text{-var}$ radical shuffle ideals are obtained from path varieties under $\mathcal{I}$.

In analogy to the Zariski topology of (real) affine space $\R^d$ resp.\ $\C^d$ (see e.g. \cite[Section~2]{gathmann}, \cite[Section~2.1]{bochnakcosteroy}),
we now introduce our new path Zariski topology of $\varltwo{\R^d}$ through defining the family of closed sets as the collection of all path varieties of $\varltwo{\R^d}$,
and so the closure operator is given by $\mathcal{V}\circ\mathcal{I}$.
On $T(\R^d)$, the $2^{-}\text{-var}$ radical operator\footnote{which is not a topological closure operator in contrast to $\mathcal{V}\circ\mathcal{I}$}
is $\mathcal{I}\circ\mathcal{V}$.

In follow up work, we will look at the same construction for path subspaces of $\varltwo{\R^d}$, for example the bounded variation paths, or for the larger rough path spaces, and will then write $\mathcal{V}^\mathcal{X}(M)$ for a given path space $\mathcal{X}$ for clarity.

We give a first example connecting classical algebraic geometry with our path Zariski topology.
For any $X:\,[0,T]\to\R^d$,
we call $X_T-X_0\in\R^d$ the increment of $X$.
\begin{theorem}
 The set $\incin{M}$ of paths in $\varltwo{\R^d}$ with increments in $M\subseteq\mathbb{R}^d$ is a path variety if and only if $M$ is a point variety.
\end{theorem}
\begin{proof}
 If $M$ is a point variety, and $(p_i)_i$ is a set of polynomials generating the corresponding polynomial ideal,
 then $\incin{M}=\mathcal{V}((\varphi(p_i))_i)$ by the shuffle identity.

 If $\incin{M}$ is a path variety, then if we intersect it with the path variety of linear paths $V^{\text{linear}}$,
 then $\incin{M}\cap V^{\text{linear}}$ is the path variety of linear paths with increment in $M$,
 and by Proposition \ref{prop:suboflinear} this implies that $M$ is a point variety.
\end{proof}

\subsection{Differences to classical algebraic geometry}\label{sec:difftoclassical}

For a single path $X$, the closure $\overline{\{X\}}=\mathcal{V}(\mathcal{I}(\{X\}))$ is the set of all paths tree-like equivalent to $X$.
However, we can recover a topology where singleton sets are closed by looking at the quotient of $\varltwo{\R^d}$ under tree-like equivalence.
The quotient $\varltwogroup{\R^d}:=\{\overline{\{X\}}|X\in\varltwo{\R^d}\}$ has a group structure
given by the concatenation product $\overline{\{X\}}\sqcup\overline{\{Y\}}:=\overline{\{X\sqcup Y\}}$
and inverse given by time reversal $\big(\overline{\{X\}}\big)^{-1}:=\overline{\big\{\overleftarrow{X}\big\}}$,
and is called the reduced path group.
The reduced path of $X$ is the up to reparametrization unique path $\check{X}$ that is tree-like equivalent to $X$ and such that $t\mapsto\sigma(\check{X})_t$ is either constant or injective (see \cite{BGLY16}).

Another new aspect is the fact that we are dealing with a necessarily infinite dimensional setting.
 If a path variety $V$ is finite dimensional (a notion that we define in Section \ref{sec:dim}),
 then $\mathcal{I}(V)$ is infinitely generated as a shuffle ideal, see Theorem \ref{thm:finitegeneratedinfinitedim}.
 Also, if we would just look at the path varieties corresponding to finitely generated ideals,
 we wouldn't get a topology on $\varltwo{\R^d}$
 since an infinite intersection of these can result in a path variety correspondonding to an infinitely generated ideal.

 If $V$ is infinite dimensional, then $\mathcal{I}(V)$ can still be finitely generated or infinitely generated. For example, the path variety of loops $\mathcal{V}(\word{1},\dots,\word{d})$ is infinite dimensional and its ideal finitely generated, while the path variety of all paths with reduced paths in the $x_1=0$ subplane is infinite dimensional with infinitely generated ideal (see Lemma \ref{lem:x_1=0}).

 Furthermore, there are path varieties which are a countable union of disjoint path subvarieties. See Example \ref{ex:circleloops} and more generally Theorem \ref{thm:concpowers}.

 Also, the classical real Nullstellensatz doesn't hold in general because we don't allow for all grouplike elements.

 \subsection{Path varieties are analytically closed}

Point varieties are also closed in the euclidean topology.
We get a similar statement for path varieties.

\begin{theorem}
 Let $p\in[1,2)$.
 If $(X_n)_n$ are paths of bounded $p$-variation in a path variety $V$ converging in $p$-variation norm against a $p$-variation path $X$,
 then $X\in V$.
\end{theorem}
\begin{proof}
 Like in the proof of Theorem \ref{thm:bdvarconv},
 this follows immediately from the convergence of the iterated integrals according to \cite[Proposition~6.11]{FrizVictoir10}.
\end{proof}

For reference, we also state the following for bounded variation paths,
i.e. paths $X$ with $\|X\|_{1\text{-var}}<\infty$.

\begin{theorem}\label{thm:bdvarconv}
   If $(X_n)_n$ are bounded variation paths in a path variety $V$ converging uniformly and with a joint variation bound towards a bounded variation path $X$,
   then $X\in V$.
  \end{theorem}
  \begin{proof}
   For any $x\in\mathcal{I}(V)$,
   we have by \cite[Proposition~6.12]{FrizVictoir10}
   that $0=\langle\sigma(X_n),x\rangle$ converges to $\langle\sigma(X),x\rangle$,
   which thus vanishes too.
  \end{proof}

\section{Concatenation}
\label{sec:conc}

\begin{lemma}
 Let $M, N$ be subsets of $\varltwo{\R^d}$ and $M$ such that it contains all tree-like equivalent paths.
 Then $M\sqcup N$ and $N\sqcup M$ contain all tree-like equivalent paths.
\end{lemma}
\begin{proof}
 If $X$ is tree-like equivalent to a path $Y$ in $M\sqcup N$, then there are $Z\in M$, $W\in N$ such that $Y=Z\sqcup W$.
 Then $X\sqcup\overleftarrow{W}\in M$ as it is tree-like equivalent to $Z$ because $X\sqcup\overleftarrow{W}\sqcup \overleftarrow{Z}=X\sqcup\overleftarrow{W}\sqcup W\sqcup \overleftarrow{Y}=X\sqcup \overleftarrow{Y}$.
 Thus $X=X\sqcup\overleftarrow{W}\sqcup W\in M\sqcup N$.
 This proves $M\sqcup N$ contains tree-like equivalent paths.
 The proof for $N\sqcup M$ is completely analogous.
\end{proof}
This in particular shows that
\begin{equation*}
 \overline{\{X\sqcup Y\}}=\overline{\{X\}}\sqcup\overline{\{Y\}}=\overline{\{X\}}\sqcup Y=X\sqcup \overline{\{Y\}}.
\end{equation*}

The following theorem is equivalent to the observation that $Z\mapsto X\sqcup Z$ and $Z\mapsto Z\sqcup X$ are Zariski homeomorphisms for all $X\in\varltwo{\R^d}$.
\begin{theorem}\label{thm:path_conc_variety}
 Let $V$ be a path variety and $X\in\varltwo{\R^d}$. Then $X\sqcup V$ and $V\sqcup X$ are path varieties.
\end{theorem}
\begin{proof}
 We have $Y\in X\sqcup V$ if and only if $\overleftarrow{X}\sqcup Y\in V$, if and only if
 \begin{align*}
  0&=\langle\sigma(\overleftarrow{X}\sqcup Y),x\rangle=\langle\antipode\sigma(X)\conc\sigma(Y),x\rangle=\langle\antipode\sigma(X)\otimes\sigma(Y),\Delta_{\conc} x\rangle=\sum_{(x)}^{\conc}\langle\antipode\sigma(X),w_1\rangle\langle\sigma(Y),w_2\rangle\\
  &=\langle\sigma(Y),\sum_{(x)}^{\conc}\langle\antipode\sigma(X),w_1\rangle w_2\rangle
 \end{align*}
 for all $x\in\mathcal{I}(V)$.
 Thus $X\sqcup V=\mathcal{V}(\{\sum_{(x)}^{\conc}\langle\antipode\sigma(X),w_1\rangle w_2|x\in\mathcal{I}(V)\})$.
 Analogously, $V\sqcup X=\mathcal{V}(\{\sum_{(x)}^{\conc}w_1\langle\antipode\sigma(X), w_2\rangle|x\in\mathcal{I}(V)\})$.
\end{proof}
Since a path variety $V$ contains tree-like equivalent paths,
we have $X\sqcup V=\overline{\{X\}}\sqcup V$ and $V\sqcup X=V\sqcup \overline{\{X\}}$.

Note that in general, concatenation of two path varieties does not need to be a path variety. Under a meaningful definition of semialgebraic path sets, the concatenation of two path varieties should be a semialgebraic path set though, see Section \ref{sec:semialg}.

\begin{theorem}\label{thm:conc_varieties}
 Let $V$ be a sub path variety of $\varltwo{\R^d}$. Then the set of paths $X$ in $\mathbb{R}^d$ such that $X^{\sqcup n}\in V$ is a path variety, and the set of paths $(X_1,\ldots,X_n)$ in $\mathbb{R}^{nd}$ such that $X_1\sqcup\dots\sqcup X_n\in V$ is also a path variety.
\end{theorem}
\begin{proof}
 This follows from the fact that the signature of $X^{\sqcup n}$ is contained in the signature of $X$, and the signature of $X_1\sqcup\dots\sqcup X_n$ is contained in the signature of $(X_1,\dots,X_n)$. To be precise,
 \begin{equation*}
  \langle\sigma(X^{\sqcup n}),x\rangle=\sum_{(x)}^{\conc n}\langle\sigma(X),w_1\rangle\cdots\langle\sigma(X),w_{n}\rangle=\langle\sigma(X),\sum_{(x)}^{\conc n}w_1\shuffle\cdots\shuffle w_{n}\rangle
 \end{equation*}
 and
  \begin{equation*}
  \langle\sigma(X^1\sqcup\dots\sqcup X^n),x\rangle=\sum_{(x)}^{\conc n}\langle\sigma(X_1),w_1\rangle\cdots\langle\sigma(X_n),w_{n}\rangle=\langle\sigma(X_1,\dots,X_n),\sum_{(x)}^{\conc n}\iota_1^\top w_1\shuffle\cdots\shuffle\iota_n^\top w_{n}\rangle
 \end{equation*}
 where $\iota_i$ is linear embedding of $\mathbb{R}^d$ into the $i$-th slot of $\mathbb{R}^{nd}$.
\end{proof}

\section{Halfshuffle varieties}
\label{sec:halfshuffle}

A path variety $V$ is called a $\succ$ path variety if $\mathcal{I}(V)$ is a $\succ$-ideal.
\begin{lemma}\label{lem:x_1=0}
 The set of all paths whose reduced path lies in the $x_1=0$ hyperplane is a path variety with ideal the $\succ$-ideal generated by $\word{1}$.
\end{lemma}
\begin{proof}
 The signatures of all paths lying in the $x_1=0$ subspace are exactly those who do not contain the letter $\word{1}$, i.e. they are characterized by $\langle \sigma(X),w\rangle=0$ for all words $w$ containing $\word{1}$.

 Indeed, with $Q:\mathbb{R}^n\to\mathbb{R}^n$ the linear projection on the $x_1=0$ hyperplane, if $\langle\sigma(X),w\rangle=0$ for all words $w$ containing $\word{1}$, then $\langle\sigma(QX),x\rangle=\langle\sigma(X),Q^\top x\rangle=\langle \sigma(X),x\rangle$ for all $x\in T(\mathbb{R}^n)$, implying that $\sigma(X)$ is the signature of the path $QX$ lying in the $x_1=0$ hyperplane.
 Conversely, if $X$ lies in the $x_1=0$ hyperplane, then $\langle\sigma(X),w\rangle=\langle\sigma(QX),w\rangle=\langle\sigma(X),Q^\top w\rangle=0$ for all words $w$ containing $\word{1}$.

 Furthermore, by the Chen-Chow theorem \cite{FrizVictoir10}[Theorem 7.28] and linear embedding of $\mathbb{R}^{n-1}$, for any word $x\in T(\mathbb{R}^n)$ not containing the letter $\word{1}$, there is a path $Y$ in the $x_1=0$ subspace such that $\langle\sigma(Y),x\rangle\neq 0$ .

 The statement follows by the fact that the $\succ$-ideal generated by the letter $\word{1}$ is indeed linearly spanned by all words containing the letter $\word{1}$, since the right halfshuffle of two words where at least one contains the letter $\word{1}$ is a linear combination of words containing the letter $\word{1}$, and each word containing the letter $\word{1}$ can be written as the left bracketed right halfshuffle of its letters, one of which is $\word{1}$.
\end{proof}

\begin{theorem}\label{thm:halfshuffle_variety}
 Whenever a set of paths $M$ contains all left subpaths of reduced paths, $\mathcal{I}(M)$ is a $\succ$-ideal.

 In particular, whenever a path variety $V$ contains all left subpaths of reduced paths, it is a halfshuffle path variety.

 Whenever $I$ is a halfshuffle ideal, $\mathcal{V}(I)$ contains all left subpaths of reduced paths.
\end{theorem}

\begin{proof}
 Let $M$ be a set of paths containing all left subpaths of reduced paths. Since this means the zero path is in $M$, we must have $\langle \emptyword,x\rangle=0$ for all $x\in\mathcal{I}(M)$, thus $\mathcal{I}(M)\subseteq  T^{\geq 1}(\R^d)$.
 Let $x\in\mathcal{I}(M)$. Then $\langle\sigma(X)_t,x\rangle=0$ for all reduced paths $X\in\mathcal{I}(V)$ and all $t\in[0,T]$. Thus for $y\in T^{\geq 1}(\R^d)$, $\langle X,x\succ y\rangle=\int_0^T\langle \sigma(X)_t,x\rangle d \langle\sigma(X)_t,y\rangle=0$. Since the reduced path of a path has the same signature, we have $\langle Y,x\succ y\rangle=0$ for any path in $M$, and thus $x\succ y\in \mathcal{I}(M)$. Since $\mathcal{I}(M)$ is a shuffle ideal, we also have $x\shuffle y\in\mathcal{I}(M)$ and thus $y\succ x\in\mathcal{I}(M)$.

Whenever we have a reduced path $X$ in a path variety $\mathcal{V}(I)$ where $I$ is a halfshuffle ideal and contains $x$, the path augmented by $\langle\sigma(X)_t,x\rangle$ in the first component, which we call $Y$, will be in the variety corresponding to the halfshuffle ideal generated by $\word 1$.
 Indeed, by \cite{colmenarejopreiss20}[Theorem 8], with $B:\,\word{1}\to x, \word{i}\to\word{i-1}$, we have that $\Lambda_B$ maps the ideal generated by $\word{1}$ to the ideal generated by $x$, and thus $\langle\sigma(Y),y\rangle=\langle\sigma(X),\Lambda_B y\rangle=0$ for any $y$ in the halfshuffle ideal generated by $\word{1}$.

 Since $X$ is reduced, $Y$ is also reduced, and so by Lemma \ref{lem:x_1=0}, this implies that $Y$ lies in the $x_1=0$ subspace,
and we get that $\langle\sigma(X)_t,x\rangle=0$ for all $t$.
 As this argument works for any $x$ in the halfshuffle ideal $I$,
 we have that any left subpath of $X$ is an element of $\mathcal{V}(I)$.
 \end{proof}

 \begin{remark}
The first part of the Theorem does not hold for rough paths.  Indeed, look at the set $M$ of rough paths such that $\langle\mathbf{X}_t,\word{1}\rangle=0$ for all $t$. $M$ contains all left subpaths of reduced paths, but also $\exp_{\conc}(t(\word{12}-\word{21}))$, so $\mathcal{I}(M)$ is not a halfshuffle ideal, since $\word{1}\in\mathcal{I}(M)$ but $\word{12}\notin\mathcal{I}(M)$.
\end{remark}

 \begin{corollary}
 Whenever a set of paths $M$ contains all right subpaths of reduced paths, $\mathcal{I}(M)$ is a $\prec$-ideal.

 Whenever $I$ is a $\prec$-ideal, $\mathcal{V}(I)$ contains all right subpaths of reduced paths.
\end{corollary}
\begin{proof}
 A set of paths $M$ contains all right subpaths of reduced paths if and only if $\overleftarrow{M}$ contains all left subpaths of reduced paths. We have $\mathcal{I}(\overleftarrow{M})=\antipode\mathcal{I}(M)$ and $\mathcal{V}(\antipode I)=\overleftarrow{\mathcal{V}(I)}$. Using Theorem \ref{thm:halfshuffle_variety}, the corollary then follows from the fact that $\antipode(x\prec y)=\antipode(x)\succ\antipode(y)$ and bijectivity of the antipode, meaning that $\antipode I$ is a $\prec$-ideal if and only if $I$ is a $\succ$-ideal.
\end{proof}
 \begin{corollary}
 Whenever a set of paths $M$ contains all subpaths of reduced paths, $\mathcal{I}(M)$ is a $\succ$-ideal and a $\prec$-ideal.

 Whenever $I$ is a $\succ$-ideal and a $\prec$-ideal, $\mathcal{V}(I)$ contains all subpaths of reduced paths.
\end{corollary}

\begin{proof}
 Suppose $M$ contains all subpaths of reduced paths. Then by the previous theorems, $\mathcal{I}(M)$ is a $\succ$-ideal and a $\prec$-ideal.

 Suppose $I$ is a $\succ$-ideal and a $\prec$-ideal. Then by the previous thereoms, $\mathcal{V}(I)$ contains all left subpaths and all right subpaths of reduced subpaths. But any subpath can be represented as a left subpath of a right subpath, so $\mathcal{V}(I)$ must contain all subpaths.
\end{proof}

 \begin{corollary}\label{cor:pathsinvariety}
  Let $p:\mathbb{R}^n\to\mathbb{R}^m$ be a polynomial map with $p(0)=0$. Let $I$ be the halfshuffle path variety generated from the $\varphi(p_i)$, $i=1,\dots,m$. Then $\mathcal{V}(I)$ is the path variety of all paths whose reduced paths with start point in $0$ lie in the zero locus of $p$.
 \end{corollary}
 \begin{proof}
  Let $M$ be the set of paths whose reduced paths lie in the zero locus of $p$. Obviously, then $\langle\sigma(X),\varphi(p_i)\rangle=0$ for $X\in M$, and since all left subpaths lie in the zero locus too, by Theorem \ref{thm:halfshuffle_variety} we know that $\mathcal{I}(M)$ must contain $I$. Thus $M$ is contained in $\mathcal{V}(I)$.

  On the other hand,
  for any path $X$ in $\mathcal{V}(I)$ we must have $\langle\sigma(X),\varphi(p_i)\rangle=0$ by definition of $I$,
  and since $I$ is a halfshuffle ideal,
  by Theorem \ref{thm:halfshuffle_variety} this must also hold for all left subpaths of reduced paths in $\mathcal{V}(I)$,
  meaning all reduced paths of paths in $\mathcal{V}(I)$ lie in the zero locus of $p$,
  and thus $\mathcal{V}(I)$ is contained in $M$.
 \end{proof}
 \begin{remark}
  Corollary \ref{cor:pathsinvariety} allows us to easily define rough paths living on an affine point variety
  by demanding that their signature is orthogonal to all $x\in\mathscr{I}(\varphi(p_i),i)$,
  where the $(p_i)_i$ are generators of the ideal corresponding point variety.
  Indeed this is justified by the fact that all geometric rough paths are limits of smooth paths under a suitable topology,
  and so it makes sense to demand that geometric rough paths living on an affine point variety should be those
  which are limits of smooth paths living on that point variety.
  This is an important application that will be explored in future work,
  also in comparison to existing notions of rough paths on manifolds.
 \end{remark}

 \begin{example}\label{ex:circleloops}
 The set $\pathsin{\partial B_1(1,0)}$ of paths $X$ in $\varltwo{\R^2}$ such that $X-X_0$ lies up to tree-like parts on the unit
circle with center in $(1,0)$ is a path variety.
As the circle's equation is given by $x^2-2x+y^2=(x-1)^2+y^2-1=0$,
the variety is characterized by the ideal $\langle\word{11}-\word{1}+\word{22}\rangle_{\succ}$.
Besides the generator, examples of the elements of the ideal are
\begin{equation*}
 \word{111}-\word{11}+\word{221},\,2\cdot\word{111}-\word{11}+\word{122}+\word{212},\,\word{112}-\word{12}+\word{222},\,\dots
\end{equation*}
Note that $\mathcal{I}(\pathsin{\partial B_1(1,0)})$, the set of all relations for $\pathsin{\partial B_1(1,0)}$, could possibly be larger than \mbox{$\langle\word{11}-\word{1}+\word{22}\rangle_{\succ}$}.
A sub path variety of containing countably infinitly many reduced paths is given by
restricting to loops, i.e. $\pathsin{\partial B_1(1,0)}\cap\loops{1}$.
\end{example}
 Note that if $V=\mathcal{V}(\mathscr{I}_{\succ}(\varphi(p_i),i))$ is the path variety of all paths whose reduced path with start point $0$ lies in the zero locus of the polynomial map $p$, then $\overleftarrow{V}=\mathcal{V}(\mathscr{I}_{\prec}(\mathcal{A}\varphi(p_i),i))$ is the variety of all paths whose reduced path with end point $0$ lies in the zero locus of the polynomial map $p$.

 It remains open for now whether the set of all paths whose reduced paths lies in a given point variety, regardless where on the point variety the starting or endpoint is, does form a path variety.
 It does form a semialgebraic path set according to Definition \ref{def:semialg_conc} though,
 as $\mathcal{V}(\mathscr{I}_{\prec}(\mathcal{A}\varphi(p_i),i))\sqcup\mathcal{V}(\mathscr{I}_{\succ}(\varphi(p_i),i))$ is the the set of all paths whose reduced paths lies in the connected component around $0$ of the zero locus of $p$,
 and then we can just take the union over the finitely many connected components.

 Let us finally summarize how the left and right halfshuffles
 allow us to ensure that certain algebraic relations hold
 not only for the full paths,
 but for certain subpaths as well.
 \begin{corollary}
  Let $S\subseteq T^{\geq 1}(\mathbb{R}^d)$. Then
  \begin{enumerate}
   \item $\mathcal{V}(\mathscr{I}_{\succ}(S))$ is the path variety of all $X$ such that $\langle \sigma(Y),x\rangle=0$ for all $x\in S$ and all left subpaths $Y$ of the reduced path of $X$,
   \item $\mathcal{V}(\mathscr{I}_{\prec}(S))$ is the path variety of all $X$ such that $\langle \sigma(Y),x\rangle=0$ for all $x\in S$ and all right subpaths $Y$ of the reduced path of $X$,
   \item $\mathcal{V}(\mathscr{I}_{\succ}(S)\cup \mathscr{I}_{\prec}(S))=\mathcal{V}(\mathscr{I}_{\succ}(S))\cap \mathcal{V}(\mathscr{I}_{\prec}(S))$ is the path variety of all $X$ such that $\langle \sigma(Y),x\rangle=0$ for all $x\in S$ and all left or right subpaths $Y$ of the reduced path of $X$,
   \item $\mathcal{V}(\mathscr{I}_{\succ,\prec}(S))$ is the path variety of all $X$ such that $\langle \sigma(Y),x\rangle=0$ for all $x\in S$ and all subpaths $Y$ of the reduced path of $X$.
  \end{enumerate}

 \end{corollary}
\begin{proof}
 Analogous to the proof of the previous corollary.
\end{proof}

\section{Varieties of paths contained in some hypersurface}
\label{sec:rank_varieties}

\begin{theorem}
 Let $x_1,\dots,x_n\in T^{\geq 1}(\mathbb{R}^d)$ and $k\leq n$. Then
 \begin{enumerate}
  \item The set of all paths $X$ such that
  \begin{equation*}
  \dim\{(\langle\sigma(Y),x_1\rangle,\dots,\langle\sigma(Y),x_n\rangle)|Y\text{ is a left subpath of the reduced path of }X\}\leq k
  \end{equation*}
  is a path variety.
  \item The set of all paths $X$ such that
  \begin{equation*}
  \dim\{(\langle\sigma(Y),x_1\rangle,\dots,\langle\sigma(Y),x_n\rangle)|Y\text{ is a right subpath of the reduced path of }X\}\leq k
  \end{equation*}
  is a path variety.
  \item The set of all paths $X$ such that
  \begin{equation*}
  \dim\{(\langle\sigma(Y),x_1\rangle,\dots,\langle\sigma(Y),x_n\rangle)|Y\text{ is a subpath of the reduced path of }X\}\leq k
  \end{equation*}
  is a path variety.
 \end{enumerate}

\end{theorem}
\begin{proof}
Choose linear maps $(b_l)_{l=1}^\infty$ such that $(b_l(x))_l$ spans the $\succ$-ideal generated by $x$ for all
$x\in T^{\geq 1}(\mathbb{R}^d)$.

Assume WLOG that $x_1,\dots, x_n$ are linearly independent. Then $X$ satisfies the dimension condition from (1) if and only if there are $n-k$ independent linear combinations $y_1,\dots, y_{n-k}\in\spann\{x_1,\dots,x_n\}$ such that $\langle\sigma(Y),y_i\rangle=0$ for any $i=1,\dots, n-k$ and any left subpath $Y$ of the reduced path of $X$.
By , this holds if and only if $\langle\sigma(X),z\rangle=0$ for all $z$ in the $\succ$-path variety generated by the $y_i$. Thus these $y_i$ exist if and only if the dimension of the span of $\{(\langle\sigma(X),b_l(x_1)\rangle)_l,\dots,(\langle\sigma(X),b_l(x_1)\rangle)_l\}$ in the space of real sequences is lower or equal $k$. But this is the case if and only if all $(k+1)\times (k+1)$ minors of the infinite matrix $\langle\sigma(X),b_l(x_i)\rangle)_{il}$ are zero. But all these minors can be represented as $\sigma(X)$ values for shuffle polynomials in the $b_l(x_i)$. So the claim is proven for (1).

The proves for (2) and (3) are analogous, just choosing linear maps $(c_l)_{l=1}^\infty$, $(d_l)_{l=1}^\infty$ such that $(c_l(x))_l$, $(d_l(x))_l$ span the $\prec$-ideal generated by $x$ resp.\ the $\prec,\succ$-ideal generated by $x$ for all $x$.
\end{proof}

As a first application, let us look at paths contained in a sphere.
\begin{corollary}\label{cor:circlearcs}
 The set of paths in $\mathbb{R}^d$ with reduced path in some $d-1$-sphere with arbitrary radius or in a hyperplane is a path variety.
\end{corollary}
\begin{proof}
 Choose $n=d+1$, as $x_1=\word{1},\dots,x_d=\word{d},x_{d+1}=\word{11}+\dots+\word{dd}$ and $k=d$ in part (1) of the previous theorem.
\end{proof}

\begin{example}
 In $d=2$, the reduced paths that are either linear or lie on some circle can be characterized by the vanishing of the $3\times 3$ minors of the infinite matrix
 \begin{equation*}
  \begin{pmatrix}
   \word{1} & \word{11} & \word{11} & \word{12} & \word{21} & \word{111} & \dots\\
   \word{2} & \word{21} & \word{12} & \word{22} & \word{22} & \word{211} &\dots \\
   x & x\succ\word{1} & \word{1}\succ x & x\succ \word{2} & \word{2} \succ x & (x\succ \word 1)\succ \word 1 & \dots
  \end{pmatrix}
 \end{equation*}
 where $x:=\word{11}+\word{22}$,
 a lowest level example of which is
 \begin{align*}
  &\word{1}\shuffle \word{21}\shuffle (\word{1}\succ x)+\word{11}\shuffle \word{12} \shuffle x+\word{11}\shuffle \word{2}\shuffle (x\succ 1) - \word{11}\shuffle \word{21}\shuffle x-\word{11}\shuffle \word{2}\shuffle (1\succ x)-\word{1}\shuffle \word{12}\shuffle (x\succ \word{1})\\
  &=0.
\end{align*}
This vanishes because
\begin{equation*}
 \begin{pmatrix}
  \word{11}+\word{11} \\
  \word{21}+\word{12} \\
  x\succ \word{1}+\word{1}\succ x
 \end{pmatrix}
 =
 \begin{pmatrix}
  \word{1}\shuffle\word{1}\\
  \word{2}\shuffle\word{1}\\
  x\shuffle \word{1}
 \end{pmatrix},
\end{equation*}

thus the sum of the second and third columns in the above matrix is always a multiple of the first column.

So we reduce our attention to
 \begin{equation*}
  \begin{pmatrix}
   \word{1} & \word{11} & \word{12} & \word{111} & \dots\\
   \word{2} & \word{21} & \word{22}  & \word{211} &\dots \\
   x & x\succ\word{1}  & x\succ \word{2} & (x\succ \word 1)\succ \word 1 & \dots
  \end{pmatrix}
 \end{equation*}
 where the unique lowest level $3\times3$ minor is given by
 \begin{align*}
  &\word{1}\shuffle \word{21}\shuffle (x\succ \word{2})+\word{11}\shuffle \word{22} \shuffle x+\word{12}\shuffle \word{2}\shuffle (x\succ \word{1}) - \word{12}\shuffle \word{21}\shuffle x-\word{11}\shuffle \word{2}\shuffle (x\succ \word{2})-\word{1}\shuffle \word{22}\shuffle (x\succ \word{1})\\
  &=-2\cdot\word{111122}+2\cdot\word{111221}+\word{112112}+\word{112121}+2\cdot\word{112222}+\word{121112}-\word{121211}-2\cdot\word{122111}-\word{122122}\\
  &\quad\,-\word{122212}-\word{211121}-\word{211211}-2\cdot
  \word{211222}-\word{212122}+\word{212221}+2\cdot\word{221111}+\word{221212}+\word{221221}\\
  &\quad\,+2\cdot\word{222112}-2\cdot\word{222211}
\end{align*}
\end{example}

Still, we don't know whether given a radius $r$, the set of paths with reduced path in a $d-1$-sphere of radius $r$ forms a path variety.

We can now answer a question by Bernd Sturmfels and Carlos Améndola on how we can describe the set of paths lying in some algebraic surface of given degree by the signature.

\begin{corollary}
 The set of paths in $\mathbb{R}^d$ with reduced path in some algebraic hypersurface of degree lower or equal $m$ is a path variety.
\end{corollary}
\begin{proof}
 Choose $n=md$, as $x_i$ the shuffle monomials in the letters up to degree $m$,  and $k=md-1$ in part (1) of the previous theorem.
\end{proof}

This of course also applies to degree $1$ hypersurfaces, the hyperplanes.
In fact, we get a general result for subspaces.

\begin{corollary}
 The set of paths in $\mathbb{R}^d$ with reduced paths in some $m$-dimensional subspace is a path variety.
\end{corollary}
\begin{proof}
 Just choose $n=d$, $x_i=\word{i}$ and $k=m$ in part (1) of the  previous theorem.
\end{proof}

\begin{example}
 \begin{equation*}
  \begin{pmatrix}
   \word{1} & \word{11} & \word{12} & \word{13} & \word{14} & \word{111} & \dots \\
   \word{2} & \word{21} & \word{22} & \word{23} & \word{24} & \word{211} & \dots \\
   \word{3} & \word{31} & \word{32} &\word{33} &\word{34} &\word{311} & \dots \\
   \word{4} &\word{41} & \word{42} & \word{43} &\word{44} &\word{411} & \dots
  \end{pmatrix}
 \end{equation*}
 Then we obtain a characterization of linear paths through $\R^4$ by the $2\times 2$ minors of this infinite matrix.
 A lowest level example is
 \begin{equation*}
  \word{1}\shuffle\word{21}-\word{11}\shuffle\word{2}=\word{211}-\word{112}.
 \end{equation*}
 Similarly, the paths contained in a two-dimensional plane are characterized by the $3\times 3$ minors, e.g.\
 \begin{align*}
  &\word{1}\shuffle\word{21}\shuffle\word{32}+\word{11}\shuffle\word{22}\shuffle\word{3}+\word{12}\shuffle\word{2}\shuffle\word{31}-\word{12}\shuffle\word{21}\shuffle\word{3}-\word{11}\shuffle\word{2}\shuffle\word{32}-\word{1}\shuffle\word{22}\shuffle\word{31}\\
  &=\word{11223}-\word{11322}-\word{12213}+\word{12312}-\word{21123}+\word{21321}+\word{22113}-\word{22311}+\word{31122}\\
  &\quad\,-\word{31221}-\word{32112}+\word{32211}
 \end{align*}

\end{example}

\begin{proposition}\label{prop:suboflinear}
 Let $V$ be a sub path variety of the path variety of linear paths $V^{\text{linear}}$. Then $V=V^{\text{linear}}\cap\incin{M}$ for some point variety $M$.
\end{proposition}
\begin{proof}
 Since for a linear path $X$ we have $\sigma(X)=\exp_{\conc}(\sum_{\word{i}=\word1}^{\word d}\langle\sigma(X),\word{i}\rangle\word{i})$, we can express any component $\langle\sigma(\cdot),x\rangle$ of the signature as a polynomial $p_x$ in the increment.
 Thus if $V$ is a sub path variety of $V^{\text{linear}}$,
 $X$ is in $V$ if and only if it is in $V^{\text{linear}}$ and $p_x(X(T)-X(0))=0$ for all $x\in\mathcal{I}(V)$.
 Thus $V=V^{\text{linear}}\cap\incin{M}$ with $M=\{v\in\mathbb{R}^d|p_x(v)=0\,\forall x\in\mathcal{I}(V)\}$.
\end{proof}

\begin{remark}
 This implies that $\bigcup_{n\in\mathbb{N}}\overline{\{X_n\}}$ for $(X_n)_{n\in\mathbb{N}}$ a family of distinct linear paths is not a sub path variety of $\varltwo{\R^d}$.
\end{remark}

\section{Connection with Varieties of Signature Tensors}\label{sec:homogeneous}

In \cite{AFS18},
the authors use a different approach.
They look at examples of linear path spaces $\mathcal{X}$ such that for any fixed level $k$,
the signature map $\sigma^{k}:\,\mathcal{X}\to (\R^d)^{\otimes k}$ will be a polynomial map.
The examples considered are
\begin{itemize}
 \item polynomial paths up to a certain degree,
 \item piecewise linear paths up to a certain number of segments,
 \item axis-parallel paths of a particular shape.
\end{itemize}

Thus, the image $\sigma^{(k)}(\mathcal{X})$ is automatically a semialgebraic subset of the $k$-level tensors $\sigma^{k}:\,\mathcal{X}\to (\R^d)^{\otimes k}$.

\begin{proposition}
Whenever $I$ is homogeneous, $\mathcal V(I)$ is stable under rescaling.

Whenever $M$ is stable under rescaling, $\mathcal I(M)$ is homogeneous.
\end{proposition}

\begin{proof}
If $I$ is homogeneous, then $\langle\sigma(\lambda X),x\rangle=\lambda^{|x|}\langle\sigma(X),x\rangle$ for all homogeneous $x\in I$, ao $\mathcal V(I)$ is stable under rescaling.

If $M$ is stable under rescaling and $x\in\mathcal I(M)$, then for all $X\in M$ we have $f(\lambda)=\langle\sigma(\lambda X),x\rangle=0$ for all $\lambda$, so also the derivatives $f^{(n)}(0)$ are zero, which are multiples of $\langle\sigma(\lambda X),x_n\rangle$ for $x_n:=\proj_n x$ the level $n$ part of $x$, so all $x_n$ are elements of $\mathcal I(M)$. So the latter is homogeneous.
\end{proof}

Whenever $x\in\mathcal{I}(\mathcal{X})$ is homogeneous,
it will always appear as a linear form in $J_{|x|}$.  Additionally, it appears as a degree $m$ polynomial in $J_k$ if and only if $|x|=km$ and $x$ can be expressed as a linear combination of shuffles of elements of $T^k(\R^d)$.

Conversely, any homogeneous polynomial $p\in J_k$ of degree $m$ via turning products into shuffles becomes a level $km$ element of $\mathcal{I}(\mathcal{X})$.

So the two approaches are a priori equivalent.
However, our new approach reveals the halfshuffle structure. Indeed, in all three cases of polynomial paths, piewewise linear paths and axis-parallel paths, the class of paths considered is stable under restriction, so contains both left and right subpaths.
Thus, the ideals $\mathcal{I}(\mathcal{X})$ are $\prec,\succ$-ideals.
This should lead to a much better understanding of varieties of signature tensors.

Furthermore, the present article suggests the study of varieties of signature tensors for further classes of paths, for example circle arcs in $\R^2$ through Corollary \ref{cor:circlearcs}.

\section{Path varieties stable under concatenation}
\label{sec:subsemigroups}

In this section, we explore how being a path variety stable under concatenation is actually a very strong requirement implying stability under time-reversal and inclusion of what we call admissible roots,
and give several algebraic characterizations.
In essence, Theorem \ref{thm:concpowers} and Theorem \ref{thm:subHopfalg} reduce everything to the well-known duality between subalgebras (here of $T((\R^d))$) and  coideals (here of $T(\R^d)$),
but we'll go through that in detail.

We call a linear subspace $C\subseteq T(\R^d)$ a coideal if
\begin{equation*}
 \Delta_{\conc} C\subseteq C\otimes T(\R^d)+T(\R^d)\otimes C.
\end{equation*}
Note that in contrast to ideals also being subalgebras, coideals are not in general subcoalgebras,
it is the other way around,
subcoalgebras are always coideals.

\begin{proposition}
 If $C\subseteq T(\R^d)$ is a coideal,
 then $\mathcal{V}(C)$ is stable under concatenation.
\end{proposition}
\begin{proof}
 Let $X,Y\in\mathcal{V}(C)$ and $x\in C$.
 Then, using Chen's identity,
 \begin{equation*}
  \langle\sigma(X\sqcup Y),x\rangle
  =\langle\sigma(X)\otimes\sigma(Y),\Delta_{\conc} x\rangle
  =0,
 \end{equation*}
 since $\Delta_{\conc} x\in C\otimes T(\R^d)+T(\R^d)\otimes C$
 and both $\sigma(X)$ and $\sigma(Y)$ lie in $C^\perp:=\{x\in T((\R^d))|\langle x,y\rangle=0\,\forall y\in C\}$ by assumption.
\end{proof}

Note that there is in general no such thing as a coideal generated by a family of elements of $T(\R^d)$,
as the intersection of coideals does not need to be a coideal again.
However, a coideal $C$ can be minimal with respect to a set $W\subset T(\R^d)$
if it contains $W$ and there is no strict subcoideal of $C$ which still contains $W$.

For example, both $\spann\{\word{1},\word{12}-\word{21}\}$ and $\spann\{\word{2},\word{12}-\word{21}\}$ are minimal coideals with respect to the signed area $\word{12}-\word{21}$,
so there is no such thing as a coideal generated by $\word{12}-\word{21}$.

Since we will furthermore later learn in this section that the ideal corresponding to a concatenation stable path variety is always a coideal,
this means that if we want to look at a maximal path variety for which the signed area vanishes and which is stable under concatenation,
we have multiple ways to choose from,
two of which are additionally requiring the $\word{1}$ increment to be zero or the requiring the $\word{2}$ increment to be zero.
Indeed, it is geometrically easy to see that the signed area is additive under concatenation if we fix any axis for the increment.

Of course, we can ask this question for any path variety $V$:
What are the maximal subvarieties of $V$ stable under concatenation?
And this is, as we will see,
equivalent to asking for the minimal coideals containing $\mathcal{I}(V)$.

\begin{figure}
\includegraphics[width=0.95\textwidth]{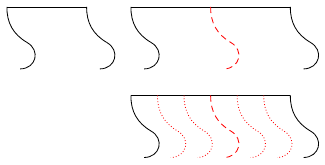}
\caption{Top left: an infinitely divisible path $X$, top right: $X^{\sqcup 2}$, bottom right: $X^{\sqcup 2}$ subdivided into $X_3^{\sqcup 6}$.}
\end{figure}

We call a path $X\in\varltwo{\R^d}$ infintely divisible if for any $n\in\mathbb{N}$,
there is $X_n\in\varltwo{\R^d}$ with
$X_n^{\sqcup n}\treelike X$.
The infinitely divisible paths are in fact exactly the conjugations $\overleftarrow{Y}\sqcup A\sqcup Y$
of linear paths $A$ with arbitrary paths $Y$.

We say that a path $X$ has an admissible power (or root) $r\in\mathbb{R}$ if there is a path $Y$ such that $\sigma(Y)=\exp_{\conc}(r\log_{\conc}(\sigma(X)))$.
\begin{theorem}\label{thm:infdivi}
 If a path $X$ is infinitely divisible, then there is a linear path $A$ such that $X\treelike \overleftarrow{Y}\sqcup A\sqcup Y$
\end{theorem}
\begin{proof}
 If $X\treelike\overleftarrow{Y}\sqcup A\sqcup Y$ for some linear path $A$,
 then $(\overleftarrow{Y}\sqcup \tfrac{1}{n}A\sqcup Y)^{\sqcup n}\treelike X$
 as $(\tfrac{1}{n}A)^{\sqcup n}\cong A$
 for a linear path $A$.

 Let now $X:\,[0,T]\to\R^d$ be an arbitrary infinitely divisible path with $X_n^{\sqcup n}\treelike X=:X_1$,
 where we may assume that the $X_n:\,[0,T_n]\to\R^d$ are reduced.
 Take the tree-like part of $X_n^{\sqcup 2}$,
 and separate it at the concatenation point $X_n(T_n)$,
 to identify it as $Z_n\sqcup\overleftarrow{Z_n}$, where then $Z_n$ is reduced again as a subpath of $X_n$.
 Then, there are paths $A_n$ and $B_n$ such that $X_n=A_n\sqcup Z_n=\overleftarrow{Z_n}\sqcup B_n$.

 If we assume that the starting point of $B_n$ comes in time order strictly later than the end point of $A_n$,
 then there is a non-constant reduced path $C_n$ such that $X_n=A_n\sqcup C_n\sqcup B_n$, $Z_n=C_n\sqcup B_n$ and $\overleftarrow{Z_n}=A_n \sqcup C_n$.
 But this implies $\overleftarrow{C_n}$ to be a subpath of $C_n$, or $C_n$ to be a subpath of $\overleftarrow{C_n}$.
 Either of these implies the other,
 and then that $C_n$ is a strict subpath of itself, which is a contradiction to Chen's identity,
 or that $C_n=\overleftarrow{C_n}$,
 which means $C_n$ is constant, but we assumed it to be non-constant.

 Thus,
 the starting point of $B_n$ comes in time order earlier than the end point of $A_n$, or both are equal in time,
 so we have that
 $X_n=\overleftarrow{Z_n}\sqcup X'_n\sqcup Z_n$
 for some reduced path $X'_n$.
 Since $(X'_n)^{\sqcup 2}$ and $X'_n\sqcup Z_n$ are reduced
 (because we already removed the full tree-like part $Z_n\sqcup \overleftarrow{Z_n}$ from $X_n^{\sqcup 2}$),
 we have that $\overleftarrow{Z_n}\sqcup (X'_n)^{\sqcup k}\sqcup {Z_n}$ for arbitrary $k$ is reduced.
 The identity
 \begin{equation*}
  \overleftarrow{Z_n}\sqcup (X'_n)^{\sqcup n}\sqcup {Z_n}\treelike (\overleftarrow{Z_{n}}\sqcup X'_{n}\sqcup Z_{n})^{\sqcup n}\treelike \overleftarrow{Z_1}\sqcup X_1'\sqcup Z_1
 \end{equation*}
 then reveals that $\overleftarrow{Z_n}\sqcup (X'_n)^{\sqcup n}\sqcup {Z_n}\cong \overleftarrow{Z_1}\sqcup X_1'\sqcup Z_1$, as they are both reduced.
 Furthermore, once again since $(X'_n)^{\sqcup 2}$ and $(X'_1)^{\sqcup 2}$ are both reduced,
 we have that both $Z_n\sqcup\overleftarrow{Z_n}$
 and $Z_1\sqcup\overleftarrow{Z_1}$
 are up to reparametrization and translation  the unique maximal tree-like excursion of
 \begin{equation*}
  \overleftarrow{Z_n}\sqcup (X'_n)^{\sqcup n}\sqcup {Z_n}\sqcup \overleftarrow{Z_n}\sqcup (X'_n)^{\sqcup n}\sqcup {Z_n}\cong \overleftarrow{Z_1}\sqcup X'_1\sqcup {Z_1}\sqcup \overleftarrow{Z_1}\sqcup X'_1\sqcup {Z_1}
 \end{equation*}
 So $Z_n\cong Z_1$ and finally $(X'_n)^{\sqcup n}\cong X'_1$

 Now $X':\,[0,T']\to\R^d$ must be linear.
 Otherwise, assume WLOG $X'_n(0)=X'(0)=0$, and take a point $p$ not on the line segment $[0,X'(T')]$. Then each $X'_n$ would have to contain a path from some point on $[0,X_n'(T'_n)]$ to some point on the parallel line containing $p$.
 So $X'$ must travel between these non-indentical parallel lines infinitely often.
 This is a contradiction to $X'$ being continuous.
\end{proof}

In general,
we say say that $r\in\R$ is an admissable power of $X$
if there is $Y\in\varltwo{\R^d}$ with $\sigma(Y)=\exp_{\conc}(r\log_{\conc}\sigma(X))$.

\begin{lemma}[Commutation Lemma]\label{lem:commutation}
 If $X\sqcup Y\treelike Y\sqcup X$,
 then both $\check{X}$ and $\check{Y}$ are conjugations of parallel linear paths,
 or $\check{X}\cong\overleftarrow{Z}\sqcup C^{\sqcup k} \sqcup Z$ and $\check{Y}\cong\overleftarrow {Z}\sqcup C^{\sqcup m}\sqcup Z$
 for paths $Z,C$ and $k,m\in\mathbb{Z}$.
\end{lemma}
\begin{proof}
 Assume WLOG that $X$ and $Y$ are reduced,
 and let $Z$ be the right subpath inclusion largest path
 such that
 $\check{X}\cong\overleftarrow{Z}\sqcup X'\sqcup Z$ and $\check{Y}\cong\overleftarrow {Z}\sqcup Y'\sqcup Z$.
 Then $X'$ and $Y'$ still commute, as
 \begin{equation*}
  X'\sqcup Y'\treelike Z\sqcup X\sqcup Y\sqcup \overleftarrow Z\treelike Z\sqcup Y\sqcup X
  \sqcup \overleftarrow Z\treelike Y'\sqcup X'.
 \end{equation*}

 Assume $X'$ and $Y'$ have a one-sided overlap,
 i.e.\ WLOG $X'\cong L\sqcup B$, $Y'\cong\overleftarrow{B}\sqcup R$
 such that $L\sqcup R$ and $R\sqcup L$ are reduced.
 Then
 \begin{equation*}
  \overleftarrow{B}\sqcup R\sqcup L\sqcup B = Y'\sqcup X'\treelike X'\sqcup Y'\treelike L\sqcup R,
 \end{equation*}
 so $L\sqcup R$ has an overlap with itself,
which is a contradiction to $L \sqcup R$ and $R\sqcup L$ being reduced.

Assume $X'$ and $Y'$ have a two-sided overlap, but neither $\overleftarrow{X'}$ is a subpath of $Y'$ nor $\overleftarrow{Y'}$ is a subpath of $X'$.
Then $X'\cong L_1\sqcup B\cong\overleftarrow{W}\sqcup R_1$
and $Y'\cong L_2\sqcup W\cong \overleftarrow{B}\sqcup R_2$,
and
\begin{equation*}
 \overleftarrow B\sqcup R_2\sqcup L_1\sqcup B\treelike \overleftarrow W\sqcup R_1\sqcup L_2\sqcup W.
\end{equation*}
Due to the assumption, the conjugations with $B$ and $W$ cannot be fully absorbed when reducing the paths,
and so there is a common conjugation by a common right subpath of $B$ and $W$,
which is a contradiction to the assumption that the maximal common conjugation has been removed through $Z$.

Assume $X'$ and $Y'$ have a two-sided overlap, and $\overleftarrow{X'}$ is a subpath of $Y'$ or $\overleftarrow{Y'}$ is a subpath of $X'$.
We consider WLOG the case that $\overleftarrow{X'}$ is a subpath of $Y'$.
Then, we consider instead the potential overlaps of the commuting paths $\overleftarrow{X'}$ and $Y'$.
By going through the prior two cases in that constellation again,
it can be excluded that they have a one-sided overlap,
and if they have a two-sided overlap,
then $X'$ must be a subpath of $Y'$ or $Y'$ is a subpath of $X'$.
The latter can be excluded, as it would lead to the contradiction that $\overleftarrow{X'}$ is a subpath of $X'$.
So $\overleftarrow{X'}$ is a left subpath of $Y'$ and $X'$ is a right subpath of $Y'$,
and thus $Y'\cong\overleftarrow{X'}\sqcup A_1 \sqcup X$ for some $A_1$.
But then $A_1 X\cong\overleftarrow{X'}\sqcup A_1\sqcup X'\sqcup X'$,
so $\overleftarrow{X'}\sqcup A_1$ is a strict subpath of $A_1\sqcup X'$, so $\overleftarrow{X'}$ is a strict subpath of $A_1\cong \overleftarrow{X'}\sqcup {A_2}$,
so $\overleftarrow{X'}^{\sqcup 2} \sqcup A_2$ is reduced and a strict subpath of $\overleftarrow{X'}\sqcup A_2\sqcup X'$,
and we can continue this arbitrarily far to reveal that $\overleftarrow{X'}$ is infinitely often contained in $A_1\sqcup X'$,
which is a contradiction.

So, after possibly performing a time reversal on one of the paths, WLOG, $X'\sqcup Y'\cong Y'\sqcup X'$ is reduced.
Then $X'$ is a subpath of $Y'$, or $Y'$ is a subpath of $X'$.
Assume, again WLOG, that $X'$ is a subpath of $Y'$.
Then let $a\in\mathbb{N}$ be the maximum number such that $(X')^{\sqcup a}$ is a left subpath of $Y'$,
and let $X'_2$ denote the remainder, $Y'\cong(X')^{\sqcup k}\sqcup X'_2$.
Then $X'_2$ is a right subpath of $X'$ that commutes with $X'$ and does not overlap with it,
so it is also again a left subpath $X'$.
Continuing this procedure, we construct a chain of left subpaths $X'_n,\dots,X'_2,X'$ of each other and of $Y'$,
such that always $X'_b=(X'_{b+1})^{\sqcup a_b}\sqcup X'_{b+2}$.
If this ever terminates without a further remainder, then $X'=(X'_n)^{\sqcup k}$, $Y'=(X'_n)^{\sqcup m}$ for some $k,m\in\mathbb{N}$ and $X'_n$ the leftmost remainder in the chain.
If it doesn't terminate, we get an infinite chain of remainders which are contained to strictly increasing powers in $Y'$,
and in those powers cover a strictly increasing left subpath of $Y'$.
By the same argumentation as at the end of the proof before,
some left subpath of $Y'$ containing $X'$ must be linear.
But then since $X'$ and $Y'$ commute,
we get that $Y'$ is linear too,
and parallel to $X'$.

\end{proof}

\begin{theorem}\label{thm:irrationalroot}
 If a path $X$ has an irrational root, then it is infinitely divisible.
\end{theorem}
\begin{proof}
 By the commutation lemma, $X$ and its irrational root $Y$,
 which commute due to $\exp_{\conc}(v)\conc\exp_{\conc}(rv)=\exp_{\conc}(rv)\conc\exp_{\conc}(v)$ for all Lie series $v$,
 are either of the form $\check{X}\cong\overleftarrow{Z}\sqcup C^{\sqcup k} \sqcup Z$ and $\check{Y}\cong\overleftarrow {Z}\sqcup C^{\sqcup m}\sqcup Z$,
 which is a contradiction to $Y$ being an irrational root,
 or $X$ and $Y$ are both conjugations of linear paths,
 and thus infinitely divisible.
\end{proof}

\begin{theorem}\label{thm:concpowers}
 Let $K\subseteq\mathbb{Z}$ be infinite. Then, if the path $\check{X}$ is not infinitely divisible, we have
 \begin{equation*}
  \overline{\{X^{\conc n}|n\in K\}}=\{Z\treelike Y^{\conc n}|n\in\mathbb{Z}\},
 \end{equation*}
 where $Y$ is the shortest root of $\check{X}$.

 If $\check{X}$ is infinitely divisible, we have
 \begin{equation*}
  \overline{\{X^{\conc n}|n\in K\}}=\{Z\treelike \check{X}^{\conc r}|r\in\mathbb{R}\}.
 \end{equation*}
\end{theorem}
\begin{proof}
 First of all, the right hand side expressions do form path varieties,
 as by the previous theorems they do contain all paths with signatures of the form $\exp_{\conc}(rv)$ for arbitrary $r\in\mathbb{R}$ where $v$ is the log-signature of $X$.
 Indeed, these are given by all the paths $Y$ satisfying the equations
 \begin{equation*}
  \langle\sigma(Y),c_v\rangle\langle\sigma(X),c_w\rangle=\langle\sigma(X),c_v\rangle\langle\sigma(Y),c_w\rangle,
 \end{equation*}
 for arbitrary coordinates of the first kind $c_v,c_w$,
 so the path varieties are those corresponding to the shuffle ideal generated by the elements $\langle\sigma(X),c_w\rangle c_v-\langle\sigma(X),c_v\rangle c_w$.

 Let $N$ be the smallest number in $\mathbb{N}$ such that $\proj_N\log_{\conc}\sigma(X)\neq 0$ and $x\in\proj_N T(\R^d)$ such that $\langle\sigma(X),x\rangle\neq 0$ (or take $x$ to be a coordinate of the first kind that doesn't vanish). Then $r=\frac{\langle\sigma(\check{X}^{\conc r}),x\rangle}{\langle\sigma(X),x\rangle}$ for all values of $r$ such that $\check{X}^{\conc r}$ exists as a path. Thus $\{\sigma(\check{X}^{\conc r})|r\in\R\text{ admissable}\}$ is polynomially parametrized by $\langle\sigma(\check{X}^{\conc r}),x\rangle$ and so there can't be a proper subset that is infinite such that the corresponding paths form a path variety.
\end{proof}

\begin{figure}
 \includegraphics[width=0.95\textwidth]{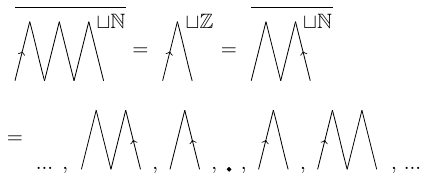}
 \caption{Example of the variety of admissable powers of a non infinite divisible path.
 The dot in the middle symbolizes the constant path, the zeroth power.
 Note that for simplicity, this graphic describes the situation in the reduced path group $\varltwogroup{\R^d}$, where we don't have tree-like excursions. }
\end{figure}

\begin{example}
 For $d=2$, the first coordinates of the first kind are
 $c_{\word{1}}=\word{1}$, $c_{\word{2}}=\word{2}$ and $c_{\word{12}}=\tfrac{1}{2}(\word{12}-\word{21})$.
 For our path with signature $\exp_{\conc}(\word{1}+4\cdot\word{2})\conc\exp_{\conc}(\word{1}-4\cdot\word{2})$,
 we have $\langle\sigma(X),\word{1}\rangle=2$, $\langle\sigma(X),\word{1}\rangle=0$ and $\langle\sigma(X),\tfrac{1}{2}(\word{12}-\word{21})\rangle=-2$,
 where the latter is the signed area.
 Thus, the first relations are given by $\word{2}$ and $\word{12}-\word{21}+2\cdot\word{1}$.
\end{example}

An immediate consequence is the following.

\begin{corollary}\label{cor:algsubsemigroups}
 Any path variety that is stable under concatenation is stable under time inverse and taking admissable roots.
\end{corollary}

In other words, any algebraic subsemigroup of the group of reduced paths is in fact a subgroup stable under taking admissable roots.

\begin{example}
 Denote the arrangement of seven lines and one circle in Figure \ref{fig:circlelines} by $M$.
 Then $M$ is an algebraic subset of $\R^2$,
 and so $\pathsin{M}$ is a path variety,
 and so is $\pathsin{M}\cap\mathscr{L}_1$,
 the loops contained in $\pathsin{M}$.
 Now $\pathsin{M}\cap\mathscr{L}_1$ is stable under concatenation, and is up to tree-like equivalence precisely given by the concatenation group generated from $X,Y,Z$.
\end{example}

\begin{figure}
 \includegraphics[width=0.7\textwidth]{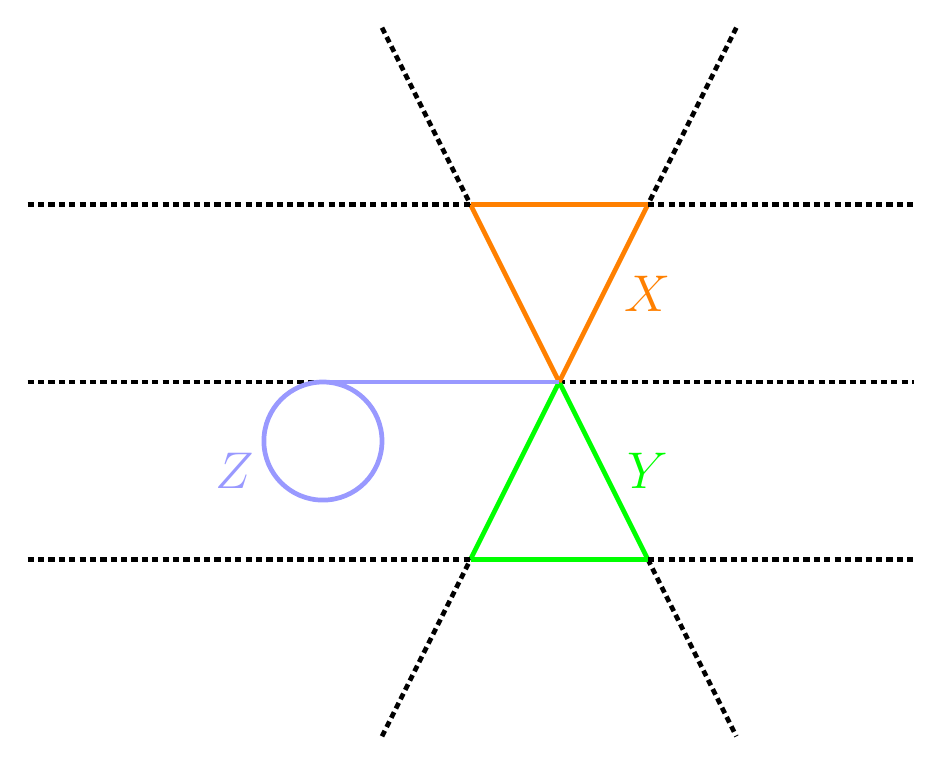}
 \caption{The concatenation group generated by the paths $X,Y,Z$ forms a path variety.}\label{fig:circlelines}
\end{figure}

Let the analytical topology on $T((\R^d))$ be the sequential one
given by $x_n\to x$ if and only if all finite truncations $\proj_{\leq k} x_n\to\proj_{\leq k} x$
in the euclidean topology on the finite dimensional vector spaces $T_{\leq k}(\R^d)$.
This makes $T((\R^d))$ a locally convex topological vector space,
with seminorms given by $|\langle\cdot,x\rangle|$, $x\in T(\R^d)$,
and thus its topological dual space can be identified with $T(\R^d)$.
Since the same topology is induced by reducing to the countable set of seminorms $|\langle\cdot,w\rangle|$, $w$ any word,
we even have a metrizable locally convex topological vector space.
As a consequence of Helly-Hahn-Banach\footnote{Eduard Helly (1884-1943). For his role in the development of what is mostly refered to as the Hahn-Banach theorem see e.g.\ \cite[Section~7.9]{NariciBeckenstein11}.}´
(cf.\ \cite[Theorem~4.7(b)]{rudin91}),
the analytically closed subspaces of $T((\R^d))$ are thus exactly the annihilators of subspaces of $T(\R^d)$.
Indeed, whenever we have an analytically closed subspace $V$ and $y\in T((\R^d))$ not contained in $V$,
let $f$ be the linear continuous functional on $V\oplus \R y$ such that $f(v+ry)=r$ for all $v\in V$.
Then, by Helly-Hahn-Banach, this functional $f$ can be extended to a continuous linear functional on all of $T((\R^d))$,
which is then an element of $V^\perp\subset T(\R^d)$
and so $y\notin (V^\perp)^\perp$.

If one does not want to work with extensions of Zermelo-Fraenkel,
then this can be circumvented by replacing any mentioning of analytically closed subspaces (sub Lie/Hopf algebra) in this section to mean instead explicitly sets of the form $W^\perp, W\subseteq T(\R^d)$ (the collection of which also does form the closed sets of a much coarser topology where closed sets are always subspaces).

\begin{theorem}\label{thm:concsubLie}
 A path variety is closed under concatenation if and only if it consists of all paths with signatures in the exponential of an analytically closed sub Lie algebra of $\mathfrak{g}((\R^d))$.
\end{theorem}

\begin{proof}
 Let $V$ be a path variety closed under concatenation and $L=\{\ell\in\mathfrak{g}((\R^d))|\langle\exp_{\conc}(\ell),x\rangle=0\,\forall x\in\mathcal{I}(V)\}$.
 First of all, $L$ is analytically closed,
 since $\log_{\conc}$ is analytically continuous.
 Then, by an argument analogous to the proof of Proposition \ref{prop:closedconc},
 $L$ is closed under the BCH-product $(\ell_1,\ell_2)\mapsto\log_{\conc}(\exp_{\conc}(\ell_1)\conc\exp_{\conc}(\ell_2))$.
 By Theorem \ref{thm:concpowers},
 for any $\ell\in L$ and $r\in\mathbb{R}$,
 we have $r\ell \in L$.
 Then by the BCH-formula,
 we have that $n\log_{\conc}(\exp_{\conc}((1/n)\ell_1)\conc \exp_{\conc}((1/n)\ell_2))$ converges to $\ell_1+\ell_2$,
 so $L$ is an analytically closed vector space.
 Furthermore, $2n^2\log_{\conc}(\exp_{\conc}((1/n)\ell_1)\conc \exp_{\conc}((1/n)\ell_2))-2n\ell_1-2n\ell_2$ converges to
 $[\ell_1,\ell_2]$,
 so $L$ is an analytically closed Lie algebra.

 If conversely $L$ is an analytically closed Lie algebra,
 then $V:=\{X\in\varltwo{\R^d}|\log_{\conc}\sigma(X)\in L\}$ is stable under concatenation by the BCH-formula.
 Finally, $V=\{X\in\varltwo{\R^d}|\langle\sigma(X),c_x\rangle\,\forall x\in L^\perp\}$,
 and thus a path variety,
 where $L^\perp$ is the annihilator of $L\subset T((\R^d))$ in $T(\R^d)$ and $c_x$ is the coordinate of the first kind corresponding to $x$ such that $\langle\ell,x\rangle=\langle\exp_{\conc}(\ell),c_x\rangle$ for all $\ell\in\mathfrak{g}((\R^d))$.

\end{proof}

\begin{proposition}\label{prop:closedconc}
 Let $U\subseteq\varltwo{\R^d}$ be closed under concatenation. Then $\bar U$ is closed under concatenation (and, by Corollary \ref{cor:algsubsemigroups}, under time inverse and admissable roots).
\end{proposition}
\begin{proof}
 For the porpose of this proof,
 let $\varltwo{\R^{d}}\times\varltwo{\R^d}$ be endowed with the topology such that the closed sets are all sets of the form
 \begin{equation*}
  \mathcal{V}(W):=\{(X,Y)|\langle\sigma(X)\otimes\sigma(Y),w\rangle=0\,\forall w\in W\}
 \end{equation*}
 where $W\subseteq T(\R^d)\otimes T(\R^d)$.
 Then $\sqcup:\varltwo{\R^{d}}\times\varltwo{\R^d}\to\varltwo{\R^d}$ is continuous,
 as for any sub path variety $V\subseteq\varltwo{\R^d}$ and any $x\in\mathcal{I}(V)$, we have
 \begin{equation*}
  \langle\sigma(X\sqcup Y),x\rangle=\langle\sigma(X)\otimes\sigma(Y),\Delta_{\conc} x\rangle,
 \end{equation*}
 so $\sqcup^{-1}(V)=\mathcal{V}(\Delta_{\conc}\mathcal{I}(V))$ and $\Delta_{\conc}\mathcal{I}(V)\subseteq T(\R^d)\otimes T(\R^d)$.

 Furthermore, for any $U_1,U_2\subseteq\varltwo{\R^d}$,
 we have $\bar{U}_1\times\bar{U}_2=\overline{U_1\times U_2}$.
 Indeed, $ \overline{U_1\times U_2}\subseteq\bar{U}_1\times\bar{U}_2$,
 as $\bar{U}_1\times\bar{U}_2=\mathcal{V}(\mathcal{I}(U_1)\otimes\emptyword+\emptyword\otimes\mathcal{I}(U_2))$,
 and, following an argument by \cite{Dap},
 $\bar{U}_1\times\bar{U}_2\subseteq\overline{U_1\times U_2}$,
 as for any $\sum_{i}x_i\otimes y_i\in T(\R^d)\times T(\R^d)$ such that $\sum_i\langle\sigma(X)\otimes\sigma(Y),x_i\otimes y_i\rangle=0$ for all $(X,Y)\in U_1\times U_2$,
 we have $\sum_i\langle\sigma(X),x_i\rangle y_i\in\mathcal{I}(U_2)$ for all $X\in U_1$,
 so in fact $\sum_i\langle\sigma(X)\otimes\sigma(Y),x_i\otimes y_i\rangle=0$ for all $(X,Y)\in U_1\times \bar{U}_2$,
 and then $\sum_i\langle\sigma(Y),y_i\rangle x_i\in\mathcal{I}(U_1)$ for all $X\in\bar{U}_2$,
 so in fact $\sum_i\langle\sigma(X)\otimes\sigma(Y),x_i\otimes y_i\rangle=0$ for all $(X,Y)\in \bar{U}_1\times \bar{U}_2$.

 Now, if we intersect the preimage of $\bar U$ under concatenation in $\varltwo{\R^{d}}\times\varltwo{\R^d}$ with $\bar U\times \bar U=\overline{U\times U}$,
 then we obtain a closed set due to the continuity of $\sqcup$,
 a closed set that contains $U \times U$, so it must in fact be $\bar U\times \bar U$.
\end{proof}

\begin{corollary}
 The set of lattice paths $\mathfrak{L}$ is Zariski dense in $\varltwo{\R^d}$.
\end{corollary}
\begin{proof}
 Since the concatenation of two lattice paths is again a lattice path,
 the Zariski closure $\bar{\mathfrak{L}}$ is also stable under concatenation.
 But since the path variety $\bar{\mathfrak{L}}$ includes the linear path of unit length in any standard basis vector direction,
 the analytically closed Lie algebra $\{\ell\in\mathfrak{g}((\R^d))|\langle\exp_{\conc}(\ell),x\rangle=0\,\forall x\in\mathcal{I}(\mathfrak{L})\}$
 is actually the full $\mathfrak{g}((\R^d))$
 and so $\bar{\mathfrak{L}}=\varltwo{\R^d}$.
\end{proof}

\begin{theorem}\label{thm:subHopfalg}
 If $M\subseteq\varltwo{\R^d}$ is a set of paths closed under concatenation,
 then $\bar{M}$ consists of all paths with signatures in a closed subalgebra of $(T((\R^d)),\conc)$
and $\mathcal{I}(M)$ is a Hopf ideal.
\end{theorem}
\begin{proof}
Since $\spann\sigma(M)$ is a subalgebra of $(T((\R^d)),\conc)$, we have that $\mathcal{I}(M)$ is a coideal.
 Thus $\mathcal{I}(M)^{\perp}$ is once again a subalgebra of $(T((\R^d)),\conc)$.
 By definition, $\bar{M}$ consists of all the paths with signatures in $H$.
Since $\mathcal{I}(M)$ is a coideal, $\bar{M}$ is closed under concatenation,
 and thus under time inverse due to \ref{cor:algsubsemigroups},
 $\mathcal{I}(M)=\mathcal{I}(\bar{M})$ is closed under the antipode,
 and so $H$ is closed under the (dual) antipode.
\end{proof}

\begin{remark}
Note that $H:=\mathcal{I}(M)^{\perp}\cap T(\R^d)$ is a Hopf algebra, since it is the orthogonal complement of $\mathcal{I}(M)$ in $T(\R^d)$.
However in general the analytical closure of $H$ could be strictly smaller than $\mathcal{I}(M)^{\perp}$ and does not need to contain all the signatures in $M$.
This is the case though e.g.\ when $M$ is also stable under rescaling, because then $\mathcal{I}(M)$ is homogeneous.
\end{remark}
\begin{remark}
The only use of Corollary \ref{cor:algsubsemigroups} is to show that $\mathcal{I}(M)$ is not just a biideal,
but a Hopf-ideal.
In fact, since $(T(\R^d),\shuffle,\Delta_{\conc},\antipode)$ is a commutative Hopf algebra,
any biideal is a Hopf ideal due to \cite[Theorem~1]{Nichols78}. So alternatively, we could also show Theorem \ref{thm:subHopfalg} using this general result,
and obtain Corollary \ref{cor:algsubsemigroups} as a consequence.
\end{remark}

\section{Dimension}\label{sec:dim}

A path variety $V$ is called countably reducible if there is a family of varieties $(V_n)_{n\in\mathbb{N}}$,
none of them equal to $V$,
such that $V=\bigcup_{n\in\mathbb{N}}V_n$.
Otherwise it is called countably irreducible.

Apart form this stricter notion of irreducibilty,
the definition of the dimension of a path variety is somewhat analogous to classical algebraic geometry (see e.g.\ \cite[Definition~2.8.1., Theorem~2.8.3.]{bochnakcosteroy}, \cite[Definition 2.25]{gathmann}),
however note Remark \ref{rem:irred_chain}.

\begin{definition}
The empty path variety has dimension $0$.

We say that all path varieties of the form $\overline{\{X\}}$, $X$ a path, and all varieties which are finite or countably infinite unions of them have dimension $0$.

We say that a path variety has dimension $n+1$ if it is a countably irreducible path variety such that it has a proper sub path variety of dimension $n$, and all proper subvarieties have been assigned a dimension in $\{0,\dots, n\}$.
We furthermore say that all path varieties which are finite or countable unions of dimension $n+1$ path varieties have dimension $n+1$.

If a countably irreducible path variety does not have a finite dimension, it is said to be infinite dimensional. Varieties which are unions of arbitrary cardinality of infinite dimensional variaties are also called infinite dimensional.
\end{definition}

\begin{remark}\label{rem:irred_chain}
 If a path variety is finite dimensional, then its dimension is equal to the largest $n$ such that there are non-empty countably irreducible varieties $V_1,\dots, V_n$ with
 $V_1\subsetneq \dots\subsetneq V_n\subseteq V$.
 It is open whether for an infinite dimensional path variety, there are countably irreducible varieties $V_1,\dots, V_n$ with
 $V_1\subsetneq \dots\subsetneq V_n\subseteq V$ for any $n$.
\end{remark}

The proof of the following lemma is along the lines of the proof of the Chen-Chow theorem in the version of \cite[Theorem~7.28]{FrizVictoir10}.
\begin{lemma}
 For any $n\in\mathbb{N}$, there is a polynomial map $q:\mathfrak{g}^n(\R^d)\to\varltwo{\R^d}$ that only takes values in the piecewise linear paths such that $\proj_{\leq n}\log_{\conc}\sigma(q(x))=x$ for all $x\in\mathfrak{g}^n(\R^d)$.
\end{lemma}
\begin{proof}
 We proceed by induction.
 For $n=1$, we can just take $q(x)$ to be the linear path from $0$ to $x$.
 Suppose we have constructed such a polynomial map $q_n$ for $n\in\mathbb{N}$.
 Then take a basis $(b_i)_i$ of $\mathfrak{g}_{n+1}(\R^d)$, a dual basis $(b'_i)_i$ in $T_n(\R^d)$ such that $\langle b_i,b'_j\rangle=\delta_{ij}$,
 and a piecewise linear path $X_i$ with $\proj_{\leq n+1}\log_{\conc}\sigma(X_i)=b_i$ (exists by the Chen-Chow theorem).
 Put $q_{n+1}(l):=q_n(\proj_{\leq n}l)\sqcup(\langle l,b'_i\rangle-\langle\sigma(q_n(\proj_{\leq n}l)),b'_i\rangle)\cdot X_i$.
 Then we have $\proj_{\leq k}\log_{\conc}\sigma(q(x))=x$ by Chen's identity and the BCH formula.
\end{proof}

\begin{lemma}\label{lem:biregular}
 Let $M\subseteq\R^d$ be a point variety and $V\subseteq\varltwo{\R^t}$ be a path variety
 such that there are regular maps $F:\R^d\to\varltwo{\R^t}$ and $G:\varltwo{\R^t}\to\R^d$
 with $G(F(a))=a$ for all $a\in M$ and
 $F(G(X))$ tree-like equivalent to $X$ for all $X\in V$.
 Then $V$ is countably irreducible if and only if $M$ is irreducible,
 and $V$ is of dimension $n$ if and only if $M$ is of dimension $n$.
\end{lemma}
\begin{proof}
 Since regular maps are Zariski continuous,
 $W=\overline{F(G(W))}$
 (Zariski closure here reduces to including all tree-like equivalent paths)
 is a sub path variety of $V$ if and only if $G(W)$ is a sub variety of $M$.
 This means $V=\bigcup_{n\in\mathbb{N}}W_n$ for subvarieties $W_n\subsetneq V$
 if and only  $M=G(V)=\bigcup_{n\in\mathbb{N}}G(W_n)$ for subvarieties $G(W_n)\subsetneq M$.
 Since point varieties are irreducible if and only they are countably irreducible,
 we get that $W$ is countably irreducible if and only if $G(W)$ is irreducible.
 This then also means a sub path variety of $V$ is countably irreducible if and only if the correpsonding subvariety of $M$ is irreducible,
 and by the definitions of dimensions for path and point varieties this means that $V$ is of dimension $n$ if and only if $M$ is of dimension $n$.
\end{proof}
The situation the Lemma describes means that the path variety of reduced paths $\check{V}$ in $\varltwogroup{\R^t}$
and the point variety $M$ are biregular.

In precisely the same way,
it can be seen that if $V_1\subseteq\varltwo{\R^t}$ and $V_2\subseteq\varltwo{\R^d}$ are varieties such that there are regular maps
$F:\varltwo{\R^d}\to\varltwo{\R^t}$ and $G:\varltwo{\R^t}\to\varltwo{\R^d}$
 with $G(F(X))$ tree like equivalent to $X$ for all $X\in V_2$ and
 $F(G(X))$ tree-like equivalent to $X$ for all $X\in V_1$,
 then $V_1$ and $V_2$ share their dimensionality and possible (countable) irreducibility.

\begin{theorem}\label{thm:finitegeneratedinfinitedim}
 Let $V$ be a path variety such that $\mathcal{I}(V)$ is finitely generated. Then $V$ is infinite dimensional.

 If $V$ is irreducible, then it is countably irreducible.
\end{theorem}
\begin{proof}
 Assume first $V$ is irreducible.
 Let $(x_i)_i$ be a finite family of elements generating $\mathcal{I}(V)$,
 and $k\in\mathbb{N}$ such that $x_i\in T_{\leq k}(\R^d)$ for all $i$.
 By Lemma \ref{lem:biregular},
 then $M:=q_k^{-1}(V)=\proj_{\leq k}\log_{\conc}(V)$ is irreducible too,
 and so $\overline{q_n(M)}$ is a strictly increasing sequence of countably irreducible subvarieties of $V$.
 Since furthermore $\bigcup_n q_n(M)$ is Zariski dense in $V$,
 this implies that $V$ itself is countably irreducible,
 and thus as a countably irreducible path variety countaining a strictly increasing infinite sequence of countably irreducible subvarieties,
 it is also infinite dimensional.

 If $V$ is reducible, since $\mathcal{I}(V)$ is finitely generated,
 we can just decompose it into again finitely generated primary ideals $I_1,\dots,I_n$ with $\mathcal{I}(V)=I_1\cap\dots\cap I_n$ like in classical algebraic geometry.
 Then $V=\bigcup_{i=1}^n\mathcal{V}(I_i)$
 where all irreducible $\mathcal{V}(I_i)$ are infinite dimensional,
 so $V$ is infinite dimensional.
\end{proof}

\begin{theorem}
 Let $M$ be a point variety of dimension $n$. Then $V^{\text{linear}}\cap \incin{M}$ has dimension $n$. If $M$ is non-empty, then $\incin{M}$ has infinite dimension.
\end{theorem}

\begin{proof}
 By Proposition \ref{prop:suboflinear},
  $V'$ is a countably irreducible sub path variety of
  $V^{\text{linear}}\cap \incin{M}$ if and only if there is an irreducible sub point variety $N$ of $M$ with $V'=V^{\text{linear}}\cap \incin{M}$.

  From this, we may conclude that the dimensions of $V^{\text{linear}}\cap \incin{M}$ and $M$ agree.

  If $M$ is non-empty, then $\mathcal{I}(\incin{M})$ is finitely generated by $\phi(p_i)$,
  where $(p_i)_i$ is a finite family generating the ideal corresponding to $M$.
  So by Theorem \ref{thm:finitegeneratedinfinitedim},
  this means that $\incin{M}$ is infinite dimensional.
\end{proof}

\begin{theorem}
 Let $0\in M\subset\R^{d+1}$ be a finite union of rotated and translated graphs of polynomial maps $\mathbb{R}\to\mathbb{R}^d$. Then $\pathsin{M}$ is one dimensional.
\end{theorem}

\begin{proof}
 First say $M\subset \R^{d+1}$ is a single rotated graph of a polynomial map $p:\,\R\to\R^d$,
 i.e.\ there is $R\in\mathrm{SO}_{d+1}$ such that
 $M=q(\R)$ for $q(r):=R(r,p(r))$.
 Then let $Q(r)$ denote the unique arc-length parametrized injective path from $0$ to $q(r)$ through $M$,
 and $R(X):=(R^{-1}(X_T-X_0))^{[1]}$.
 Then $Q,R$ are regular maps which satisfy the assumptions of Lemma \ref{lem:biregular}
 with $\pathsin{M}=\overline{Q(\R)}$,
 and thus $\pathsin{M}$ is one dimensional.

 Now let $M\ni 0$ be a finite union of rotated and translated graphs, wlog connected.
 Let $a_1,\dots,a_n$ be the intersection points of the individual graphs $(N_i)_i$
 (as the $N_i$ are finitely many algebraic curves,
 there only exist finitely many intersection points).
 Then we can write
 \begin{equation*}
  \pathsin{M}=\bigcup_{(b_1\dots b_k)\in S,\,N\in\{N_i,i\}\text{ containing }b_k}X(b_1,b_2)\sqcup\dots\sqcup X(b_{k-1},b_k)\sqcup\pathsin{N}
 \end{equation*}
 where $S$ is the countable set of finite sequences $(b_1,\dots,b_k)$ of $b_i\in\{0,a_1,\dots,a_n\}$
 such that for any $(b_i,b_{i+1})$, $i=1,\dots,k-1$, there
 is a unique injective arc-length parametrized path $X(b_i,b_{i+1})$ through $M$ with start point $b_i$ and end point $b_{i+1}$
 such that $\{0,a_1,\dots,a_n\}\cap X(b_i,b_{i+1}){\restriction_{(0,\ell)}}=\emptyset$ where $\ell$ is the length of $X(b_i,b_{i+1})$.
 So $\pathsin{M}$ is a countable union of one dimensional varieties,
 and thus one dimensional.
\end{proof}

\begin{conjecture}
 Let $M\subset\R^d$ be a finite union of algebraic curves. Then $\pathsin{M}$ is one dimensional.
\end{conjecture}

\section{Morphisms}
\label{sec:morphisms}

In principle, there is a simple sufficient criterion for a map to be continuous.

\begin{proposition}\label{prop:regfunc}
 Let $F:\varltwo{\R^d}\to\varltwo{\R^t}$ be such that there is a homomorphism $H:(T(\R^t),\shuffle)\to(T(\R^d),\shuffle)$ such that
 \begin{equation*}
  \langle\sigma(F(X)),x\rangle=\langle\sigma(X),H x\rangle
 \end{equation*}
 for all $x\in T(\R^t)$ and all $X\in\varltwo{\R^d}$.
 Then $F$ is Zariski continuous.

\end{proposition}

This condition is not necessary.
As in classical algebraic geometry,
there are non-polynomial maps which are Zariski continuous.
Indeed, take e.g.\ the map $X\mapsto f(\langle\sigma(X),\word{1}\rangle)\cdot Y$ for any bijective $f:\R\to\R$ and any non-constant $Y\in\varltwo{\R^d}$.
This map is always Zariski continuous, as the preimage of any $\overline{\{Z\}}$ in the target space is either empty or all $X$ with $\langle\sigma(X),\word{1}\rangle=r$ for some $r\in\R$.

In analogy to classical algebraic geometry (cf.\ e.g.\ \cite[Theorem~1.4.8.]{netzer}),
we use the sufficient criterion above actually as a definition for the stronger than Zariski continuity notion of a regular map between path spaces.
We'll leave it to future work to indroduce regular maps between arbitrary (affine) path varieties,
but this will follow the same principles.
\begin{definition}
 We call $F:\varltwo{\R^d}\to\varltwo{\R^t}$ a regular map if it satisfies the condition of Proposition \ref{prop:regfunc}.
\end{definition}

What is the challenge is to understand which shuffle homomorphisms $H$ are such that the transpose $H^\top: T((\R^d))\to T((\R^t))$ reduces to a map from the signatures of $\varltwo{\R^d}$ paths to the signatures of $\varltwo{\R^t}$ paths.

Thankfully, the condition we just defined for maps to be regular map reduces to any adjoint linear map $H$ existing,
which will then automatically be a shuffle homomorphism.
\begin{proposition}
 Let $F:\varltwo{\R^d}\to\varltwo{\R^t}$ be such that there is a linear map $H:T(\R^t)\to T(\R^d)$ such that
 \begin{equation*}
  \langle\sigma(F(X)),x\rangle=\langle\sigma(X),H x\rangle
 \end{equation*}
 for all $x\in T(\R^t)$ and all $X\in\varltwo{\R^d}$.
 Then $H$ is a shuffle homomorphism and $F$ is a regular map.
\end{proposition}
\begin{proof}
 Let $x_1,x_2\in T(\R^t)$.
 Then, by the shuffle relation,
 \begin{align*}
  \langle\sigma(X),H(x_1\shuffle x_2)\rangle&=\langle\sigma(F(X)),x_1\shuffle x_2\rangle
  =\langle\sigma(F(X)),x_1\rangle\langle\sigma(F(X),x_2\rangle\\&=\langle\sigma(X),H x_1\rangle\langle\sigma(X), H x_2\rangle=\langle\sigma(X),H x_1\shuffle H x_2\rangle
 \end{align*}
 for all $X\in\varltwo{\R^d}$.
 By \cite[Lemma~5.2]{DLPR20}, this implies $H x_1\shuffle H x_2=H(x_1\shuffle x_2)$.
\end{proof}

\begin{definition}
 $A:\varltwo{\R^d}\to \mathcal{X}$, $\mathcal{X}$ a vector space, is called a regular map if
 $A$ can be expressed as
 \begin{equation*}
  A(X)=\sum_{j=1}^r \langle\sigma(X),x_j\rangle v_j
 \end{equation*}
 for some $x_j\in T(\R^d)$ and some $v_j\in\mathcal{X}$.

 $D:\mathcal{X}\to\varltwo{\R^d}$ is called a regular map if there is an algebra homomorphism $K:\,(T(\R^d),\shuffle)\to\R[\mathcal{X}]$ such that
 \begin{equation*}
  \langle\sigma(D(v)),w\rangle=(Kw)(v)
 \end{equation*}
 for all $v\in\mathcal{X}$ and all $w\in T(\R^d)$.
\end{definition}
Note that in particular,
a regular map $f:\varltwo{\R^d}\to\R$
is of the form $f(X)=\langle\sigma(X),x\rangle$
for some $x\in T(\R^d)$,
and we call the regular maps $\varltwo{\R^d}\to\R$ regular functions.
Write $\R[\varltwo{\R^d}]$ for the ring of regular functions on $\varltwo{\R^d}$,
which is then canonically isomorphic to $(T(\R^d),\shuffle)$.

We end with the following justification of our definition of regular maps.
\begin{proposition}
 A map $F:\,\varltwo{\R^d}\to\varltwo{\R^t}$
 is a regular map
 if and only if for any $g\in\R[\varltwo{\R^t}]$,
 we have $F\circ g\in\R[\varltwo{\R^d}]$.
\end{proposition}
\begin{proof}
 If for any $g\in\R[\varltwo{\R^t}]$,
 we have $F\circ g\in\R[\varltwo{\R^d}]$,
 this means that for any $x\in T(\R^t)$,
 we have an $h(x)\in T(\R^d)$
 with
 \begin{equation*}
  \langle \sigma(F(X)),x\rangle=\langle\sigma(X),h(x)\rangle
 \end{equation*}
 $\langle\sigma(X),h(\alpha x+y)\rangle=\langle\sigma(F(X)),\alpha x+y\rangle
 =\alpha\langle\sigma(F(X)),x\rangle+\langle\sigma(F(X)),y\rangle
 =\langle\sigma(X),\alpha h(x)+h(y)\rangle$
 and thus by \cite[Lemma~5.2]{DLPR20},
 $h(\alpha x+y)=\alpha h(x)+h(y)$ for all $\alpha,x,y$,
 so $h$ is linear, and $F$ a regular map.

 If $F$ is a regular map and $g(X)=\langle\sigma(X),y\rangle$,
 then $g(F(X))=\langle\sigma(X),H x\rangle$ for $H$ the shuffle homomorphism corresponding to $F$.
\end{proof}

\subsection{Halfshuffle homomorphisms}

Not all shuffle homomorphisms between tensor algebras correspond to maps from paths to paths.
However, when we restrict to the stronger property of being a halfshuffle homomorphism, then we actually always have a corresponding path map.
\begin{theorem}
\label{thm:MpLambdaB}
 \begin{enumerate}
  \item Let $p:\mathbb{R}^d\to\mathbb{R}^t$ be a polynomial map with $p(0)=0$ and $V$ an $\mathbb{R}^t$ path variety. The set of paths $X$ such that $p(X-X_0)\in V$ is a path variety given by $\mathcal{V}(M_p\mathcal{I}(V))$.
  \item Let $B:\mathbb{R}^t\to T(\mathbb{R}^d)$ be linear and $V$ an $\mathbb{R}^t$ path variety. The set of paths $X$ such that $\Lambda_B^*X\in V$ for $(\Lambda_B^*X)_t:=(\langle\sigma(X)_t,B(\word{1})\rangle,\dots,\langle\sigma(X)_t,B(\word{t})\rangle)$ is a variety given by $\mathcal{V}(\Lambda_B\mathcal{I}(V))$.
 \end{enumerate}
\end{theorem}
\begin{proof}
 $p(X-X_0)\in V$ if and only if
 \begin{equation*}
  0= \langle \sigma(p(X-X_0)),x\rangle=\langle\sigma(X),M_p x\rangle
 \end{equation*}
 for all $x\in\mathcal{I}(V)$.
 Let $X^B:\,t\mapsto(\langle\sigma(X)_t,B(\word{1})\rangle,\dots,\langle\sigma(X)_t,B(\word{t})\rangle)$. Then, by \cite{colmenarejopreiss20}[Theorem 8], $X^B\in V$ if and only if
 \begin{equation*}
  0=\langle\sigma(X^B),x\rangle=\langle\sigma(X),\Lambda_B x\rangle
 \end{equation*}
 for all $x\in\mathcal{I}(V)$.

\end{proof}

 Note that $M_p\mathcal{I}(V)$ and $\Lambda_B\mathcal{I}(V)$ will in general not be a shuffle ideal, but needs to be extended to one.

 Also note importantly that $p(X-X_0)$ might be tree-like equivalent to a path $p(Y-Y_0)$
 even if $X-X_0$ and $Y-Y_0$ are not tree-like equivalent.

 So for example, it might be the case that $p(X-X_0)\in\pathsin{M}$,
 i.e.\ the reduced path of $p(X-X_0)$ lies in the point variety $M$,
 even if there is no $Y$ tree-like equivalent to $X$ such that $p(Y-Y_0)$ lies in $M$.

 We give some examples illustrating this intrcicacy.

 \begin{example}
  Let $p:\mathbb{R}^3\to\mathbb{R}^2$ be given by $p(x,y,z)=(x,y)$. Then $\mathcal{V}(M_p\mathcal{I}(\overline{\{\mathbf{0}\}}))$, the path variety of all $(X^{[1]},X^{[2]},X^{[3]})$ such that $(X^{[1]},X^{[2]})$ is tree-like, is strictly larger than $\mathcal{V}(\mathscr{I}_{\succ}(\varphi(p_1),\varphi(p_2)))$, the variety of paths $(X^{[1]},X^{[2]},X^{[3]})$ which are tree-like equivalent to $(0,0,X^{[3]})$.
  Consider for example $[0,2]\to \mathbb R,\,t\mapsto ((t-1)^2,0,t)$.
 \end{example}

 \begin{example}
  While any path $(X^{[1]},X^{[2]},X^{[3]})$ which is tree-like equivalent to $(0,0,X^{[3]})$ and $(X^{[1]},0,0)$ is tree-like, there are non-tree-like paths $(X^{[1]},X^{[2]},X^{[3]})$ such that $(X^{[1]},X^{[2]}), \,(X^{[1]},X^{[3]})$ and $(X^{[2]},X^{[3]})$ are all tree-like, i.e. the path variety $\mathcal{V}(M_p T^{\geq 1}(\mathbb{R}^2)\cup M_q T^{\geq 1}(\mathbb{R}^2) \cup M_r T^{\geq 1}(\mathbb{R}^2))$, where $p(x,y,z)=(x,y),\, q(x,y,z)=(x,z), \,r(x,y,z)=(y,z)$, is strictly larger than $\overline{\{\mathbf{0}\}}$.

 An example can be constructed by taking any non-tree-like loop $L_1$ in $\mathbb{R}^2$, mirroring it by some axis such that the resulting $L_2\neq \overleftarrow{L_1}$,
 then embedding $\{L_1, L_2\}$ once in some plane in $\mathbb{R}^3$, $L_1\mapsto L_1', L_2\mapsto L_2'$,
 and $\{\overleftarrow{L_2},\overleftarrow{L_1}\}$ in another non-parallel plane, $\overleftarrow{L_2}\mapsto \overleftarrow{L_2''}$, $\overleftarrow{L_1}\mapsto \overleftarrow{L_1''}$,
 such that the intersection axis of the two planes is parallel  to both embedded symmetry axes of $\{L_1', L_2'\}$ and $\{\overleftarrow{L_2''},\overleftarrow{L_1''}\}$.
 The example path will then be $L_1'\sqcup \overleftarrow{L_1''}\sqcup L_2'\sqcup\overleftarrow{L_2''}$.
 Then choose the $z$-axis as parallel to the embedded symmetry axes, and the $x$ and $y$-axis such that they half the angles between the two planes we chose for the embeddings.

 \begin{figure}
 \includegraphics[width=0.4\textwidth]{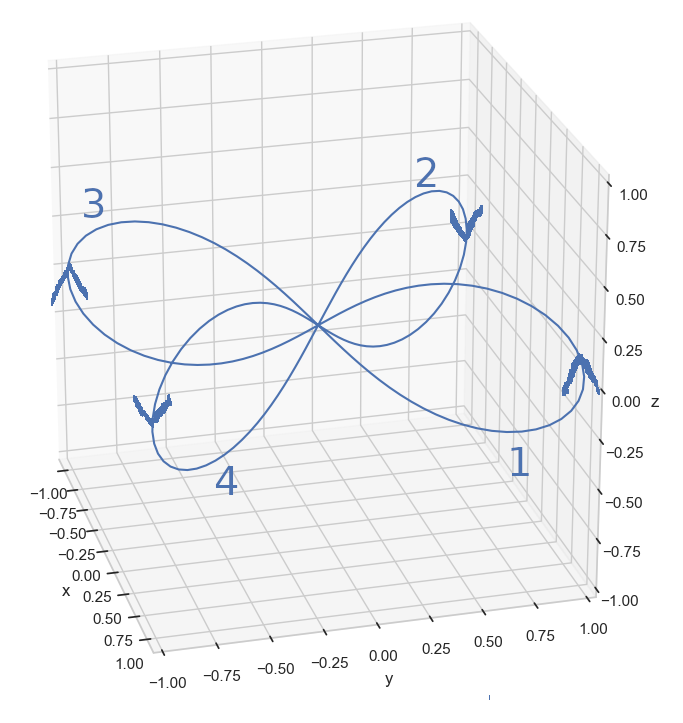}\quad
 \includegraphics[width=0.45\textwidth]{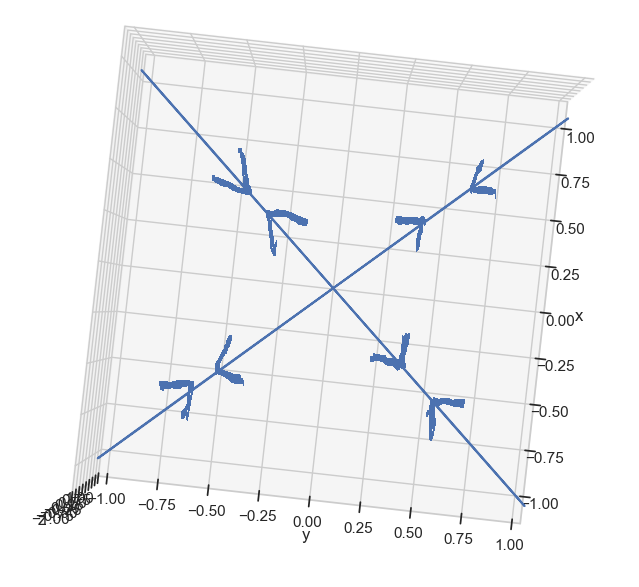}
 \caption{The path through $\R^3$ on the left side which passes through the four loops in the given order
 and in the directions indicated by the arrows
 is an example of a non-tree like path whose standard two dimensional projections are all tree-like.
 The picture on the right side for example shows the path from above, clearly tree-like in the $x$-$y$-projection.
 Note that if the path instead passes through the loops with the same directions, but in the order 1-3-2-4,
 then its $x$-$y$ and $x$-$z$-projections are still tree-like,
 but its $y$-$z$-projection is not.}
 \end{figure}
 \end{example}

 \begin{example}
  Let $p:\mathbb{R}^2\to\mathbb{R}^2$ be given by $p(x,y)=(x^2,y)$. Then $\mathcal{V}(M_p\mathcal{I}(\overline{\{\mathbf{0}\}}))$, the path variety of all $(X^{[1]},X^{[2]})$ such that $((X^{[1]}-X^{[1]}_0)^2,X^{[2]})$ is tree-like, is strictly larger than $\mathcal{V}(\mathscr{I}_{\succ}(\varphi(p_1),\varphi(p_2)))$, the variety of paths $(X^{[1]},X^{[2]})$ which are tree-like.
  Consider for example the circular path $[0,2\pi]\to\mathbb R,\, t\mapsto (\sin(t),\cos(t))$.
 \end{example}

 In general, for a polynomial map $p:\R^d\to\R^t$,
 we have that the path variety $\{X\in\varltwo{\R^d}|p(X)\text{is treelike}\}$ (cf.\ \cite[Corollary~3]{colmenarejopreiss20}) is contained in $\pathsin{\{x|p(x)=0\}}$, but the two need not be equal.

To come back to the statement from the beginning of this subsection,
indeed, any right halfshuffle homomorphism is a map of the form $\Lambda_B$,
and thus corresponds to a regular map by Theorem~\ref{thm:MpLambdaB}.

\begin{proposition}
\label{prop:halfshufflehoms}
Let $H:\,T(\R^t)\to T(\R^d)$ be linear. Then
 \begin{enumerate}
  \item If $H$ is a $\succ$ homomorphism,
  then there is a unique $B\in L(\R^t,T(\R^d))$ such that
  $H=\Lambda_B$.
  \item If $H$ is a $\prec$ homomorphism,
  then there is a unique $B\in L(\R^t,T(\R^d))$ such that
  $H=\antipode\Lambda_B\antipode$.
  \item If $H$ is such that $H(x\succ y)=H(x)\prec H(y)$ for all $x,y\in T(\R^t)$,
  then there is a unique $B\in L(\R^t,T(\R^d))$ such that $H=\antipode\Lambda_B$.
  \item If $H$ is such that $H(x\prec y)=H(x)\succ H(y)$
  for all $x,y\in T(\R^t)$,
  then there is a unique $B\in L(\R^t,T(\R^d))$ such that $H=\Lambda_B\antipode$.
 \end{enumerate}

\end{proposition}
\begin{proof}
 For (1), choose $B(\word{i}):=H\word{i}$ and then we have $H=\Lambda_B$ by them both being $\succ$ homomorphisms and the fact that $(T(\R^t),\succ)$ is generated by the letters, see \cite[Proposition~1.8]{L95}.
 Uniqueness of $B$ follows from $\Lambda_B=\Lambda_{B'}$ implying
 $B(\word{i})=\Lambda_B\word{i}=\Lambda_{B'}\word{i}=B'(\word{i})$.

 (2), (3) and (4) then follow directly by Equation \eqref{eq:antipodehalfshuffle}.

\end{proof}

\subsection{More general morphisms}
As mentioned before,
it is a difficult open problem to explicitly characterize the set of regular maps $\varltwo{\R^t}\to\varltwo{\R^d}$.
Since this seems out of reach for now, we instead develop in this section
a rich subclass of the regular maps which contains all examples that we can think about so far.

In this section,
we restrict to the spaces $\varltwoz{\R^d}$ of paths started in $0$ for simplicity,
but note that any regular map $F:\varltwoz{\R^d}\to\varltwoz{\R^t}$
together with a completely arbitrary $f:\varltwo{\R^d}\to\R^t$
forms a regular map $f+F:\varltwo{\R^d}\to\varltwo{\R^t},\,(f+F)(X)_t=f(X)+F(X-X_0)_t$
(this is simply due to the fact that the signature doesn't see the starting point),
and all regular maps $\varltwo{\R^d}\to\varltwo{\R^t}$
can be obtained in this way.

First of all,
we generalize the class of $\Lambda_B$ by allowing $B$ to depend on $X$.

\begin{lemma}
\label{lem:Lambda_BX}
 Let $B:\varltwoz{\R^d}\to L(\R^t,T(\R^d))$ be a regular map.
 Then $\Lambda_{B(\cdot)}^*:\,\varltwoz{\R^d}\to\varltwoz{\R^t}, X\mapsto \Lambda_{B(X)}^*X$ is a regular map.
\end{lemma}
\begin{proof}
Any regular map $B:\varltwoz{\R^d}\to L(\R^t, T(\R^d))$ can be expressed as
\begin{equation*}
B(X)(\word{i})=\sum_{j}\langle\sigma(X),x_{ij}\rangle y_{ij}
\end{equation*}
for some $x_{ij}, y_{ij}\in T(\R^d)$.

For $w=\word{i}_1\dots\word{i}_n$, we then have
\begin{align*}
 \langle\sigma(\Lambda^*_{B(X)}X),w\rangle&=\langle\sigma(X),\Lambda_{B(X)}w\rangle
 =\sum_{j_1,\dots,j_n}\langle\sigma(X),x_{i_1j_1}\rangle\dots\langle\sigma(X),x_{i_nj_n}\rangle\langle\sigma(X),y_{i_1j_1}\succ\dots\succ y_{i_nj_n}\rangle\\
 &=\langle\sigma(X),\sum_{j_1,\dots,j_n} x_{i_1j_1}\shuffle \dots\shuffle x_{i_nj_n}\shuffle(y_{i_1j_1}\succ\dots\succ y_{i_nj_n})\rangle
\end{align*}
and the map $\word{i}_1\dots\word{i}_n\to \sum_{j_1,\dots,j_n} x_{i_1j_1}\shuffle \dots\shuffle x_{i_nj_n}\shuffle(y_{i_1j_1}\succ\dots\succ y_{i_nj_n})$ extends linearly to all of $T(\R^t)$.
\end{proof}

Since compositions of regular maps are regular maps,
we get that any
\begin{equation*}
 \varltwo{\R^d}\to\varltwo{\R^t},\,X\mapsto \antipode^{s_1}\Lambda^*_{B_k(X)}\antipode\Lambda^*_{B_{k-1}(X)}\cdots\antipode\Lambda^*_{B_2(X)}\antipode\Lambda^*_{B_1(X)}\antipode^{s_0}X
\end{equation*}
for $k\in\mathbb{N}_0$ ($k=0$ refers to the identity $X\mapsto X$ and to time inverse $X\mapsto \antipode X$), $s_0,s_1\in\{0,1\}$, $n_i\in\mathbb{N}$ with $n_1=d$ and $n_{k+1}=t$, $B_i:\,\varltwo{\R^d}\to L(\R^{n_{i+1}},T(\R^{n_i}))$ regular maps,
is a regular map,
and we call the class of all functions of this form $\Upsilon(d,t)$.
Then $\Upsilon(d,t)\circ\Upsilon(t,e)\subseteq\Upsilon(d,e)$ is a direct consequence of Proposition \ref{prop:halfshufflehoms} and the fact that $\antipode^2=\id$.
Thus, $((\varltwoz{\R^d})_d,\Upsilon)$ is a small category and a subcategory of the small category of $((\varltwoz{\R^d})_d$ with regular maps.

The class of regular maps introduced in the following lemma is such that the path $F(X)$ is developed only from finite dimensional information about $X$.
\begin{lemma}
\label{lem:Lambda_AY}
 Let $A:\varltwoz{\R^d}\to L(\R^t,\R^n)$ be a regular map and $Y\in\varltwoz{\R^n}$.
 Then $F:\,\varltwoz{\R^d}\to\varltwoz{\R^t}, X\mapsto \Lambda_{A(X)}^*Y$ is a regular map.
\end{lemma}
\begin{proof}
Any regular map $A:\,\varltwoz{\R^d}\to L(\R^t,\R^n)$ can be expressed as
\begin{equation*}
 A(X)(\word{i})=\sum_j \langle\sigma(X),x_{ij}\rangle v_{ij}
\end{equation*}
for some $x_{ij}\in T(\R^d)$ and $v_{ij}\in\R^n$.
Remember that we identify $\R^n$ with the subspace of letters in $T(\R^n)$.

For $w=\word{i}_1\dots\word{i}_n$,
we then have

\begin{align*}
 \langle\sigma(\Lambda^*_{A(X)}Y),w\rangle&=\langle\sigma(Y),\Lambda_{A(X)}w\rangle
 =\sum_{j_1,\dots,j_n}\langle\sigma(X),x_{i_1j_1}\rangle\dots\langle\sigma(X),x_{i_nj_n}\rangle\langle\sigma(Y),v_{i_1j_1}\succ\dots\succ v_{i_nj_n}\rangle\\
 &=\langle\sigma(X),\sum_{j_1,\dots,j_n} \langle \sigma(Y),v_{i_1j_1}\succ\dots\succ v_{i_nj_n}\rangle\, x_{i_1j_1}\shuffle \dots\shuffle x_{i_nj_n}\rangle
\end{align*}
and the map $\word{i}_1\dots\word{i}_n\to\sum_{j_1,\dots,j_n} \langle \sigma(Y),v_{i_1j_1}\succ\dots\succ v_{i_nj_n}\rangle\, x_{i_1j_1}\shuffle \dots\shuffle x_{i_nj_n}$
extends linearly to all of $T(\R^t)$.
\end{proof}

To start putting everything together,
we observe that concatenation of two regular maps yields a regular map again.

\begin{lemma}
\label{lem:FsqcupG}
 Let $F,G:\varltwoz{\R^d}\to\varltwoz{\R^t}$ be regular maps. Then $X\mapsto F(X)\sqcup G(X)$ is a regular map.
\end{lemma}
\begin{proof}
 By Chen's identity and the shuffle relation,
 \begin{equation*}
  \langle\sigma(F(X)\sqcup G(X)),w\rangle=\langle\sigma(X),\sum_{(w)}^{\conc} H_1 w_1\shuffle H_2 w_2\rangle,
 \end{equation*}
 where $H_1$ and $H_2$ are the shuffle homomorphisms corresponding to $F$ and $G$ respectively.
\end{proof}

As a final incredient,
we have the result that the $\Lambda_B^*$ split nicely under concatenation,
a generalization of \cite[Corollary~4]{colmenarejopreiss20}

\begin{lemma}\label{lem:lambdasplit}
 Let $B\in L(\R^t,\R^n)$ and $X,Y\in\varltwoz{\R^n}$.
 We have
 \begin{equation*}
  \Lambda_B^*(X\sqcup Y)=\Lambda_B^*(X)\sqcup\Lambda_{\delta_X B}^*(Y)
 \end{equation*}
where $\delta_X B(\word{i}):=B(\word{i})+\sum_{(w=B(\word{i}))}^{\conc}\langle\sigma(X),w'\rangle\,w''$.
\end{lemma}
\begin{proof}
Let the time domain of $X$ be $[0,T_1]$ and that of $Y$ be $[0,T_2]$.
For $t\in[0,T_1]$,
\begin{equation*}
 \Lambda_B^*(X\sqcup Y)_t^{[i]}
 =\langle\sigma(X\sqcup Y)_t,B(\word{i})\rangle
 =\langle\sigma(X)_t,B(\word{i})\rangle
 =(\Lambda_B^*X)_{t}^{[i]}
 =(\Lambda_B^* X\sqcup \Lambda_{\delta_X B}^* Y)_t^{[i]}.
\end{equation*}

For $t\in[T_1,T_1+T_2]$, we have
\begin{align*}
 \Lambda_B^*(X\sqcup Y)_t^{[i]}
 &=\langle\sigma(X\sqcup Y)_t,B(\word{i})\rangle
 =\sum_{(w=B(\word{i}))}^{\conc}\langle\sigma(X),w_1\rangle\langle\sigma(Y)_{t-T_1},w_2\rangle
 \\&=\langle\sigma(X),B(\word{i})\rangle+\langle\sigma(Y)_{t-T_1},B(\word{i})\rangle+\sum_{(w=B(\word{i}))}^{\conc}\langle\sigma(X),w'\rangle\langle\sigma(Y)_{t-T_1},w''\rangle\\
 &=\langle\sigma(X),B(\word{i})\rangle+\langle\sigma(Y)_{t-T_1},B(\word{i})\rangle+\sum_{(w=B(\word{i}))}^{\conc}\langle\sigma(X),w'\rangle\langle\sigma(Y)_{t-T_1},w''\rangle\\
 &=\langle\sigma(X),B(\word{i})\rangle+\langle\sigma(Y)_{t-T_1},\delta_X B(\word{i})\rangle
 =(\Lambda_B^*X)_{T_1}^{[i]}+(\Lambda_{\delta_X B}^* Y)_{t-T_1}^{[i]}
 \\&=(\Lambda_B^* X\sqcup \Lambda_{\delta_X B}^* Y)_t^{[i]}.
\end{align*}
\end{proof}

As a summary of everything we did in this subsection,
we present the largest class of regular maps that we can describe so far.

\begin{theorem}
 Let $F:\,\varltwoz{\R^d}\to\varltwoz{\R^t}$ be of the form
 \begin{equation*}
  X\mapsto \Lambda^*_{A_1(X)}Y_1\sqcup\Psi_1(X)\sqcup\Lambda^*_{A_2(X)}Y_2\sqcup\Psi_2(X)\sqcup\dots\sqcup\Psi_k(X)\sqcup\Lambda^*_{A_{k+1}(X)}Y_{k+1},
 \end{equation*}
 where $k\in\mathbb{N}_0$, $Y_i\in\varltwoz{\R^{n_i}}$, $A_i:\varltwoz{\R^d}\to L(\R^t,\R^{n_i})$ a regular map, $\Psi_i\in\Upsilon(d,t)$.

 Then $F$ is a regular map.
 We denote by $\Gamma(d,t)$ the class of all functions of this form.
 Then $\Gamma(d,t)\circ\Gamma(t,e)\subseteq\Gamma(d,e)$.
\end{theorem}
\begin{proof}
 $\Gamma$ being a regular map follows directly from the Lemmata \ref{lem:FsqcupG}, \ref{lem:Lambda_BX} and \ref{lem:Lambda_AY}.

 Closedness under composition follows from
 \begin{enumerate}
  \item $\antipode(X\sqcup Y)=\antipode Y\sqcup \antipode X$,
  \item $\Upsilon$ being closed under composition,
  \item $\Lambda^*_{B(X)}(F(X)\sqcup G(X))=\Lambda^*_
  {B(X)}(F(X))\sqcup \Lambda^*_{\delta_XB(X)}(G(X))$ by Lemma
  \ref{lem:lambdasplit},
  and $X\mapsto \delta_XB(X)$ being once again a regular map,
  \item $B\circ F:\varltwoz{\R^d}\to L(\R^e,T(\R^t))$ being a regular map for any regular map $F:\varltwoz{\R^d}\to\varltwoz{\R^t}$ and any regular map $B:\varltwoz{\R^t}\to L(\R^e,T(\R^t))$,
  \item $\antipode\Lambda_{A(X)}^*=\Lambda_{A(X)}^*\antipode$ for all regular maps $A:\varltwoz{\R^d}\to L(\R^t,\R^{n_i})$
  \item $\Lambda_{A_1(X)}Y\sqcup\Lambda_{A_2(X)}Z=\Lambda_{\tilde{A}_1(X)+\tilde{A}_2(X)}(\tilde{Y},\tilde{Z})$,
  where if $Y:\,[0,T_1]\to\R^n$ and $Z:\,[0,T_2]\to\R^m$,
  we put $\tilde{A}_1(X)(v_1,v_2):=A_1(X)(v_1)$ and $\tilde{A}_2(X)(v_1,v_2):=A_2(X)(v_2)$ for $v_1\in\R^n$ and $v_2\in\R^m$, as well as $\tilde{Y}_t=Y_t$ and $\tilde{Z}_t=0$ for $t\in[0,T_1]$
  as well as $\tilde{Y}_t=Y_T$ and $\tilde{Z}_t=Z_{t-T_1}-Z_0$ for $t\in[T_1,T_1+T_2]$,
  \item It is easy to see that for regular maps $A:\varltwoz{\R^d}\to L(\R^t,\R^{n})$ and $B:\varltwoz{\R^d}\to L(\R^e,T(\R^t))$,
  we have that $\varltwoz{\R^d}\to L(\R^e,T(\R^n)),\,X\mapsto \Lambda_{A(X)}B(X)$ is a regular map again,
  and so there are $x_{ij}\in T(\R^d)$, $v_{ij}\in T(\R^n)$ with
  $\Lambda_{A(X)}B(X)(\word{i})=\sum_j\langle\sigma(X),x_{ij}\rangle \,v_{ij}$
  and thus
  \begin{align*}
   (\Lambda_{B(X)}^*\Lambda_{A(X)}^*Y)^{[i]}_t
   &=\langle\sigma(\Lambda_{A(X)}^*Y),B(X)(\word{i})\rangle
   =\langle\sigma(Y),\Lambda_{A(X)}B(X)(\word{i})\rangle
   \\&=\sum_{j}\langle\sigma(X),x_{ij}\rangle\langle\sigma(Y),v_{ij}\rangle
   =(\Lambda_{C(X)}^*Z)_t^{[i]}
   \end{align*}
   with $C(X)(\word{i})=\sum_{j}\langle\sigma(X),x_{ij}\rangle \word{(i,j)}$
   and
   $Z^{[(i,j)]}_t=\langle\sigma(Y),v_{ij}\rangle$.
  \end{enumerate}
\end{proof}

Note that this includes the identity through choosing $k=1$, $Y_1=Y_2=0$ and $\Psi_1=\id$,
as well as all $X\mapsto \Lambda^*_{A(X)}Y$ through $k=0$ and $\Psi_1(X)=0$.
Once again, $((\varltwoz{\R^d})_d,\Gamma)$ is a subcategory of the small category of $((\varltwoz{\R^d})_d$ with regular maps,
and a supercategory of $((\varltwoz{\R^d})_d,\Upsilon)$.

\begin{remark}
 If we are looking at bounded variation paths instead, we simply restrict the $Y_i$ to be of bounded variation.

 For piecewise linear paths, we have to restrict the $Y_i$ to be piecewise linear, and furthermore restrict the $B_i$ of the functions in $\Upsilon$ to only map to $L(\R^t,\mathscr{A}(\R^d))$,
 where $\mathscr{A}(\R^d)\subset T(\R^d)$ denotes the iterated areas,
 see \cite[Theorem~5.3]{DLPR20}.

 Finally, if we are looking at the space of all rough paths of arbitrarily low $p$-variation,
 then it makes sense to include the translation maps $T_v$ introduced in \cite{brunedchevyrevfrizpreiss19} into the definition of $\Upsilon$.
\end{remark}

\section{Generalized point-varieties}
\label{sec:generalized}
If we look at the halfshuffle ideal generated by $\word{1}-\word{2}-\word{12}$, this means that the equation $x(t)=\int_0^t(x(s)+1)dy(s)$ is satisfied along all reduced paths $(x,y)$ with $x(0)=y(0)=0$ in the corresponding path variety. Similar to the standard trick of deriving $f(t)e^{-t}$ to proof that $f(t)=f'(t)$ has the general solution $f(t)=ce^{-t}$,
we have
\begin{align*}
 (x(t)+1)e^{-y(t)}&=\int_0^t(x(s)+1)de^{-y(s)}+\int_0^t e^{-y(s)}dx(s)+(x(0)+1)e^{-y(0)}\\
 &=-\int_0^t(x(s)+1)e^{-y(s)}dy(s)+\int_0^t e^{-y(s)}(x(s)+1)dy(s)+1=1,
\end{align*}
where the first equality uses integration by parts,
while the second equality uses  $\int_0^t g'(h(s))dh(s)=g(h(s))-g(h(0))$ for $g(r)=e^{-r}, h(s)=y(s)$ and \eqref{eq:anazinbiel}.
Thus we see that $x=e^y-1$ is the unique solution to the initial value problem with $x(0)=0, y(0)=0$, meaning all paths with reduced paths lying in $\{(x,y)\in\mathbb{R}^2|x=e^y-1\}$ form the corresponding path variety. However, it is well known that the graph of the exponential function is not an algebraic curve.

Similarly, the halfshuffle ideal generated by $\word{1}-\word{2}-\word{12}-\word{32}$ corresponds to the IVP $x(t)=\int_0^t(x(s)+z(s)+1)dy(s)$, $y(0)=z(0)=0$ with unique solution $x=e^y-1-z$, meaning all paths with reduced paths lying in $M_2=\{(x,y,z)\in\mathbb{R}^2|x=e^y-1-z\}$ form the corresponding path variety. Again, $M_2$ is not an algebraic surface.

\begin{theorem}
Let $M$ be $\varltwo{\R^d}$ connected. If $\pathsin{M}$ is a path variety, then so is $\pathsin{M-p}$ for any $p\in M$.
\end{theorem}
\begin{proof}
 Let $X\in\pathsin{M}$ with endpoint $p$ be arbitrary. Then $X\sqcup Y\in\pathsin{M}$ if and only if $Y\in\pathsin{M-p}$.

 Now $X\sqcup Y\in\pathsin{M}$ if and only if
 \begin{align*}
  0&=\langle\sigma(X\sqcup Y),x\rangle=\langle\sigma(X)\conc\sigma(Y),x\rangle=\langle\sigma(X)\otimes\sigma(Y),\Delta_{\conc} x\rangle=\sum_{(x)}^{\conc}\langle\sigma(X),w_1\rangle\langle\sigma(Y),w_2\rangle\\
  &=\langle\sigma(Y),\sum_{(x)}^{\conc}\langle\sigma(X),w_1\rangle w_2\rangle
 \end{align*}
 for all $x\in\mathcal{I}(\pathsin{M})$.
 Thus $\pathsin{M-p}=\mathcal{V}(\{\sum_{(x)}^{\conc}\langle\sigma(X),w_1\rangle w_2|x\in\mathcal{I}(\pathsin{M})\})$.

\end{proof}

Motivated by the examples and the theorem, we have the following definition.
\begin{definition}
$M\subseteq \mathbb{R}^d$ is called a proper generalized point variety
if it is $\varltwo{\R^d}$ connected
and for $p\in M$ the set $\pathsin{M-p}$ of all paths with reduced paths lying in $M-p$ is a path variety.
$M\subseteq \mathbb{R}^d$ is called an irreducible generalized point variety
if it is a proper generalized point variety
and there are no proper subsets $M_1, M_2$ of $M$ such that $M=M_1\cup M_2$ and $\pathsin{M_1-p_1}$ as well as $\pathsin{M_2-p_2}$ are both path varieties for $p_1\in M_1$, $p_2\in M_2$.
\end{definition}

By Corollary \ref{cor:pathsinvariety}, if $M$ is a point variety, then $M$ is a proper generalized point variety if and only if $M$ is path connected,
and $M$ is an irreducible generalized point variety if and only if $M$ is an irreducible variety.

The tricky part so far is that we don't know whether for $M_1, M_2$ being proper generalized point varieties such that $M_1\cap M_2\neq \emptyset$,
we have that $M_1\cup M_2$ is again a proper generalized point variety,
i.e. $\pathsin{M_1\cup M_2-p}$ is a path variety for $p\in M_1\cup M_2$.

Furthermore, it is still open whether an arbitrary intersection of proper generalized point varieties can have infinitely many $\varltwo{\R^d}$ connected components.

However, we have some stability of the concept under arbitrary intersection in the following way.

\begin{theorem}
 If $(M_i)_{i\in I}$ is a family of proper generalized point varieties, where $I$ may be of arbitrary cardinality,
 then any $\varltwo{\R^d}$ connected component of $\bigcap_{i\in I} M_i$ is a proper generalized point variety.
\end{theorem}
\begin{proof}
 Let $C$ be an arbitrary $\varltwo{\R^d}$ connected component of $\bigcap_{i\in I} M_i$ and $p\in C$. Then
 \begin{equation*}
  \pathsin{C-p}=\pathsin{(\bigcap_{i\in I} M_i)-p}=\pathsin{\bigcap_{i\in I} (M_i-p)}=\bigcap_{i\in I} \pathsin{M_i-p},
 \end{equation*}
 and so a path variety since the family of path varieties is closed under arbitrary intersection.

\end{proof}

\section{Outlook}
\subsection{Coordinate rings}
For a given path ideal $V$, $\mathcal{I}(V)$ is a shuffle ideal, so the quotient $\R[V]:=T(\R^d)/\mathcal{I}(V)$ is once again a commutative algebra,
and analogous to affine point varieties we call $\R[V]$ the coordinate ring of $V$.

Now, if $V$ contains left subpaths, then $\mathcal{I}(V)$ is a two-sided $\succ$-ideal, so $\R[V]$ is even a Zinbiel algebra again.

If $V$ is stable under concatenation, then $\mathcal{I}(V)$ is a Hopf ideal, and then $\R[V]$ is a Hopf algebra.

Note that as an important difference to affine point varieties,
to understand the geometrical structure of $V$, we need the algebraic structure of $\R[V]$ \emph{plus} the $2^-$ radical operator on the power set of $\R[V]$, as we do not have a Nullstellensatz yet.

\subsection{Invariantization of varieties}
Let $G$ be a group acting on the path space such that there is a transpose linear action on the tensor algebra with $\langle\sigma(AX),x\rangle=\langle\sigma(X),A^\top x\rangle$ for all paths $X$ and $x\in T(\mathbb{R}^d)$.

For a path variety $V$, let $V^G$ be the set of paths $X$ such that there is $A\in G$ with $AX\in V$.

For an ideal $I$, let $I^G:=\{x\in I|A^\top x\in I\forall A\in G\}$.
Futhermore, put $I_G:=\{x\in T(\mathbb{R}^d)|\exists A\in G:\,A^\top x\in I\}$.
\begin{lemma}
 $\mathcal{I}(V^G)=\mathcal{I}(V)^G$ and $\overline{\mathcal{V}(I)^G}=\mathcal{V}(I^G)$ and $\mathcal{V}(I)_G=\mathcal{V}(I_G)$.
\end{lemma}
\begin{proof}

Then $\mathcal{I}(V^G)$ is contained in $\mathcal{I}(V)$, but also invariant, and since $\mathcal{I}(V)^G$ is the largest invariant subset of $\mathcal{I}(V)$, we have $\mathcal{I}(V^G)\subseteq\mathcal{I}(V)^G$. Let $x\in\mathcal{I}(V)^G$, then $A^\top x\in \mathcal{I}(V)$ for all $A\in G$, and thus $0=\langle \sigma(X),A^\top x\rangle=\langle\sigma(AX),x\rangle$ for all $A\in G$ and $X\in V$, meaning $x\in\mathcal{I}(V^G)$. Thus $\mathcal{I}(V^G)=\mathcal{I}(V)^G$.

Likewise, $\mathcal{V}(I^G)$ contains $\mathcal{V}(I)$, and is invariant, and since $\mathcal{V}(I)^G$ is the smallest invariant superset of $\mathcal{V}(I)$, we have $\mathcal{V}(I)^G\subseteq\mathcal{V}(I^G)$. Also, $\mathcal{I}(\mathcal{V}(I)^G)=\mathcal{I}(\mathcal{V}(I))^G\subseteq I^G$ and thus $\overline{\mathcal{V}(I)^G}\supseteq\mathcal{V}(I^G)$, but also $\overline{\mathcal{V}(I)^G}$ is the smallest path variety containing $\mathcal{V}(I)^G$, so $\overline{\mathcal{V}(I)^G}=\mathcal{V}(I^G)$.

$X\in \mathcal{V}(I_G)$
 if and only if $0=\langle\sigma(X),A^\top x\rangle=\langle\sigma(AX),x\rangle$ for all $x\in I$ and all $A\in G$,
 so if and only if $X\in\mathcal{V}(I)_G$.
\end{proof}

So if $\mathcal{V}(I)^G$ is a path variety, it is equal to $\mathcal{V}(I^G)$. However, in general it is not, not even for compact $G$.
Take for example the path variety of paths in $\mathbb{R}^2$ whose increment vector lies in some non-circular ellipse with length of semi-major axis $b$ and length of semi-minor axis $a$. Then the $O_2$ invariantization of this path variety is not a sub path variety of $\varltwo{\R^2}$, as it will be the set of paths whose increment vector has a length greater or equal $a$ and smaller or equal $b$. In this case, $\mathcal{V}(I^G)$ must necessarily be all paths in $\mathbb{R}^2$.

In the ongoing project \cite{diehllyonsnipreiss},
it is shown that if the path variety $V$ is just the set of paths being tree-like equivalent to a given path $X$,
and $G$ is a subgroup of $O_d$ acting pointwise on $\varltwo{\R^d}$,
then $V^G$, the $G$ orbit of $X$ up to tree-like equivalence,
is a path variety given by the values of the $G^\top$ invariants in $T^{\geq 1}(\mathbb{R}^d)$.

\begin{question}
 For a path variety $V$, can $V^G$ generally be described as a semialgebraic set?
\end{question}

For $G$ compact, let $\Pi_G^\top x:=\int_{G}A^\top xd\mu_G(A)$, where $\mu_G$ is the Haar measure on $G$ (this was used as a tool in \cite{diehllyonsnipreiss}).
\begin{theorem}
 For $G$ compact, we have $\mathcal{V}(I_G)=\mathcal{V}(\Pi_G^\top I)$. In particular, if $J$ is $G$ invariant, we have $\mathcal{V}(J)=\mathcal{V}(\Pi_G J)$.
\end{theorem}
\begin{proof}
If $x\in I$, then so is $x\shuffle x$.
 \begin{equation*}
  \langle\sigma(X),\Pi_G^\top (x\shuffle x)\rangle=\int_G\langle\sigma(AX),x\rangle^2 d\mu_G(A)
 \end{equation*}
 is zero if and only if $\langle\sigma(AX),x\rangle=0$ for all $A\in G$.
 So $\mathcal{V}(I)_G=\mathcal{V}(\Pi_G^\top I)$.
\end{proof}
This means that for $G$ compact, all $G$-invariant varieties can be described by a subset of the $G$-invariant tensors vanishing.

This is not true for $G$ noncompact.
For example, $\overline{\{\mathbf{0}\}}$ is $G$-invariant for any $G$,
but not described by $G$-invariant tensors vanishing for any noncompact matrix group $G$ acting pointwise on pathspace.
So the varieties described by $G$-invariant tensors vanishing forms an interesting subclass of the $G$-invariant varieties.

\subsection{Complexification and Projectification}

For complexification, we can just look at paths in $\mathbb{C}^d$ and the complex signature which is given by all complex iterated integrals of the $d$ complex components of the path. However, the complex signature will fail to characterize complex paths up to tree-like equivalence. Indeed, already on level 2, we only get $d(d-1)/2$ independent complex signature components being equivalent to $d(d-1)$ independent real valued components, in contrast to the $d(2d-1)$ independent components of the level 2 real valued signature for paths in $\mathbb{R}^{2d}$.
This issue has already been discussed in \cite{BoedihardjoGeng19}.

A single complex signature will describe a set of paths which forms a path variety in $\mathbb{R}^{2d}$ which contains an infinite set of reduced paths. However, this does not have to concern us as we mostly look at complexification for the nicer algebraic structure.

A complex valued Chen-Chow theorem follows immediately from the real valued Chen-Chow theorem \cite{FrizVictoir10}[Theorem 7.28], and so we get Hilbert's Nullstellensatz for finitely generated shuffle ideals.

\begin{question}
 Is there a variant of Hilbert's Nullstellensatz for finitely generated halfshuffle ideals over $\varltwo{\C^d}$ or bounded variation paths through $\C^d$?
\end{question}

Projective varieties are all $\mathcal{V}(I)$ for $I$ a homogeneous ideal, i.e. an ideal generated by homogeneous tensors. Projective path space is the space of all
'lines' $\{Z\treelike rX|r\in\R\}$.

Another possibility is to look at log-projective varieties.
That is varieties $\mathcal{V}(I)$ for $I$ an ideal generated by tensors which are homogeneous shuffle polynomials in the coordinates of the first kind. These are exactly the varieties such that for any path $X$ in them, all $\check{X}^{\sqcup r}$ for $r\in\mathbb{R}$ admissable are contained.
Log-projective path space is given by all log-projective 'lines' $\{Z\treelike X^{\sqcup r}|r\in\R\text{ admissible}\}$.

Finally, it seems fitting to also understand the family of halfshuffle varieties as a form of projective varieties.
The halfshuffle projective path space is then given by the set of all halfshuffle varieties which do not contain a strict halfshuffle subvariety. We conjecture that all these minimal halfshuffle varieties are of the form $\pathsin{M}$ and one dimensional.
\subsection{Semialgebraic sets}
\label{sec:semialg}
The definition of semialgebraic sets of paths is a bit tricky because in order to include complements of path varieties, we need to allow for countably infinite $\vee$ compositions of inequalities, and in order to include varieties, we need to allow for countably infinite $\wedge$ compositions of equalities, but then again we don't want to allow arbitrary countably infinite $\vee$/$\wedge$ compositions of equalities and inequalities.

One possibility (\textbf{Definition 1}) is to allow countably infinite $\vee$/$\wedge$ compositions of equalities and inequalities such that for any $n$, only finitely many of the (in)equalities can be expressed in terms of polynomials over the level $n$ truncation of the log signature.

Another possibility is an axiomatic approach.

\begin{definition}\label{def:semialg_conc}
The family of semialgebraic path sets is the smallest family of subsets of $\varltwo{\R^d}$ that contains path varieties and is stable under finite union, complement and concatenation (the family of varieties is not stable under complement and concatenation).
\end{definition}

\begin{definition}\label{def:semialg_proj}
The familiy of semialgebraic path sets is the smallest family of subsets of $\bigsqcup_{d=2}^\infty\varltwo{\R^d}$ that contains path varieties and is stable under finite union,  complement and linear projection.
\end{definition}

The semialgebraic path sets from Definition \ref{def:semialg_proj} contain those from Definition \ref{def:semialg_conc}, which is a consequence of Theorem \ref{thm:conc_varieties}. Indeed, for path varieties $V_1,V_2$, the paths $(X,Y)$ such that $X\sqcup Y\in V_2$ and $X\in \overleftarrow{V_1}$ form a path variety, so the projection of that path variety on the $Y$ part yields $V_1\sqcup V_2$. How the definitions relate to each other beyond that remains to be seen.

However, whenever we have a notion stable under concatenation, we can describe piecewise linear paths and axis parallel paths up to a certain number of steps as a semialgebraic set. If our notion of semialgebraic path sets is stable under linear projection and $\mathrm{GL}_d$ invariantization, we can describe polynomial paths up to a certain order as a semialgebraic set. Indeed, let $X_1,\dots,X_d$ be $d$ copies of the moment curve in $\mathbb{R}^n$, then taking the $\mathrm{GL}_{nd}$ orbit of $(X_1,\dots,X_d)$ and projecting onto the first $d$ coordinates will yield polynomial paths in $\mathbb{R}^d$ up to order $n$. It is an open problem whether all subpaths of algebraic curves up to a certain order can be described as a semialgebraic set.

In any case, the study of semialgebraic path sets should lead to meaningful connections with the work on "Varieties of signature tensors" \cite{AFS18}, and it's follow-up papers.

\subsection{Piecewise linear and polynomial complexity}

Given a path variety $V$, let the piecewise linear complexity be the smallest $n\in\mathbb{N}$ such that there are infinitely many non-tree like equivalent piecewise linear paths with $n$ pieces in $V$.

Similarly, let the polynomial complexity be the smallest $n\in\mathbb{N}$ such that there are infinitely many non-tree like equivalent polynomial paths of degree $n$ in $V$.

Let $f_n:\R^{d\times n}\to \varltwo{\R^d}$ be the map from $n$ vectors to the piecewise linear path consisting of these $n$ pieces.
Then $f_n^{-1}(V)$ is a point variety for all path varieties $V$, and if we only consider the topology of closed sets of this type on $\R^{d\times n}$,
then the piecewise linear complexity is the smallest $n\in\mathbb{N}$ such that $\dim f_n^{-1}(V)>0$ (this is then the dimension up to tree-like equivalence once again).

Both piecewise linear and polynomial complexity are not invariant under morphisms, however we can look at the minimal/maximal complexity under morphisms as an invariant.

\begin{figure}
 \includegraphics[width=0.4\textwidth]{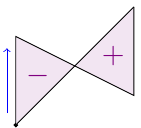}
 \caption{An example of a piecewise linear path with the least number of segments (four) such that the first two levels of the signature vanish.}\label{fig:zeroarea}
\end{figure}

As an example of particular importance, the piewewise linear complexity of $\mathcal{L}_2$,
the variety of paths with signature vanishing up to level $2$, is four, see Figure \ref{fig:zeroarea}.
Similar questions have been considered e.g.\ in \cite{GK14}, \cite[Section~1.6]{bib:Rei2018}, \cite[Section~2]{lyonsxu17} and \cite[Section~5.3]{AFS18}.

\bibliographystyle{alpha}
\bibliography{bib}

 \end{document}